\documentclass{article}

\usepackage[T2C]{fontenc}
\usepackage[utf8]{inputenc}
\usepackage[english]{babel}
\usepackage{amsmath, amssymb, amsthm}
\usepackage[all]{xy}
\newtheorem{theorem}{Theorem}
\newtheorem{lemma}{Lemma}
\newtheorem{corollary}{Corollary}
\newtheorem{proposition}{Proposition}
\newtheorem{conjecture}{Conjecture}
\newtheorem{definition}{Definition}
\newtheorem{remark}{Remark}
\newtheorem{example}{Example}
\numberwithin{equation}{section}

\usepackage{upgreek}
\usepackage{tikz}
\usepackage{verbatim}
\usepackage{hyperref}
\usetikzlibrary{calc}

\usepackage{caption}
\newcommand{\Z}{\mathbb{Z}}
\newcommand{\R}{\mathbb{R}}
\newcommand{\N}{\mathbb{N}}

\begin{document}

\date{}

\author{
Tatyana Zaitseva
\thanks{Laboratory ``High-dimensional approximation and applications'', Lomonosov Moscow State University, Moscow Center for Fundamental and Applied Mathematics, {e-mail: \tt\small
zaitsevatanja@gmail.com}} 
}

\title{Simple tiles and attractors
\thanks{
The research of the author was carried out with the support of the grant of the Government of the Russian Federation (project 14.W03.31.0031).
}}

\maketitle

\markright{Simple tiles and attractors}

\begin{abstract}

We study self-similar attractors in the space $\R^d$, i.e., self-similar compact sets defined by several affine operators with the same linear part. 
The special case of attractors when the matrix~$M$ of the linear part of affine operators and the shifts are integer, is well known in the literature due to many applications in the construction of wavelet and in approximation theory. In this case, if an attractor has measure one, it is called a tile.  
We classify self-similar attractors and tiles in case when they are either polyhedra or union of finitely many polyhedra. We obtain a complete description of the integer contraction matrices and of the digit sets for ti\-les\--pa\-ralle\-le\-pipeds and for convex tiles in arbitrary dimension. It is proved that on a two-di\-men\-si\-onal plane, every polygonal tile (not necessarily convex) must be a parallelogram. Non-trivial examples of multidimensional tiles which are a finite union of polyhedra are given, and in the case $d = 1$ their complete classification is provided. Appli\-ca\-tions to orthonormal Haar systems in $\mathbb{R}^d$ and to integer univariate tiles are considered.

\bigskip

\smallskip

\begin{flushright}
\end{flushright}

\end{abstract}

\section{Introduction}
\label{intro}

Self-similar tiles and attractors have been widely studied in the literature both intrinsically and due to applications in function theory, digital signal processing, combinatorics and approximation theory (see \cite{GrHa} -- \cite{CHM}). 
They are used in construction of multivariate Haar functions and other systems of wavelets in $\R^d$. As a rule, self-similar tiles and attractors have rather complex fractal structure. Even the simplest questions about their properties remain open, for example, the classification of connected tiles, etc. In this work we analyse the problem of characterization of ``simple'' tiles and attractors, which are polyhedra or a finite union of polyhedra. This problem is practically important not only due to the simplicity of construction of the corresponding Haar functions but also due to the fact that these functions have the maximal regularity among all possible Haar bases in $\R^d$ (regularity issue was analysed, for example, in \cite{CharProtArch} and \cite{myprev}).
The smoothness implies the fast convergence of the corresponding partial sums of wavelet expansions, that is also important in practice. 

Recall that a tile in $\R^d$ is a set of points of the form $\sum_{k = 1}^{\infty}M^{-k}b_k$ of measure one, where $M$ is an integer matrix, all of whose eigenvalues are larger than one in the absolute value (expanding matrix), all $b_k$ are from a finite set of ``digits'' in $\Z^d$. The set of digits contains $|\det M|$ elements, one from each class of equivalence of $\Z^d / M\Z^d$ (see Definition \ref{def1}
 further). A tile is a compact set whose integer shifts cover all space with one layer. To a certain extent, tile is a multidimensional generalization of the unit segment $[0, 1]$ for a ``number system'' with the matrix $M$ base. The generalization of the notion of tile is the notion of attractor. In this case, $M$ is an arbitrary expanding matrix (not necessarily integer) and ``digits'' $b_k$ are arbitrary vectors from $\R^d$. An attractor has the following characteristic property of self-similarity: $G = \bigcup_{s_k} {M^{-1}(G + s_k)}$, where the sets $M^{-1}(G + s_k)$ have intersections only of zero measure (see Definition \ref{def2} further). As a tile is also an attractor, we can say that a tile is a self-similar set whose integer shifts cover the space with  overlappings of zero measure. 

For example, the tile corresponding for the matrix $\begin{pmatrix}1 & 1 \\ -1 & 1\end{pmatrix}$ and digits $\begin{pmatrix}0 \\ 0\end{pmatrix}$, $\begin{pmatrix}1 \\ 0\end{pmatrix}$ is a so called Dragon tile (see Fig. \ref{dragon}).

\begin{figure}[h!]
\begin{minipage}[h!]{0.4\linewidth}
\centering
\includegraphics[width = 1\textwidth]{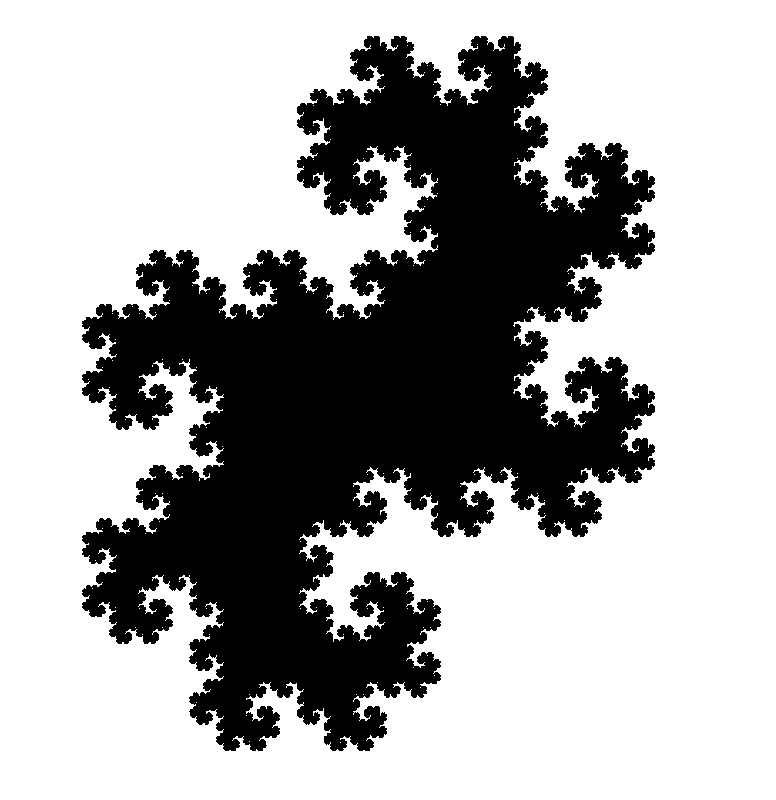}

\caption{Dragon}
\label{dragon}
\end{minipage}
\quad
\begin{minipage}[h!]{0.4\linewidth}
\centering
\includegraphics[width = 1\textwidth]{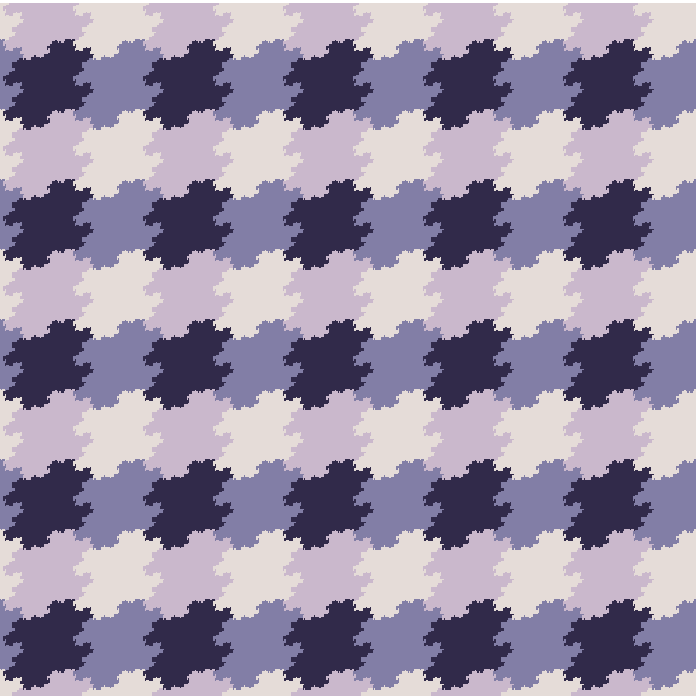}

\caption{Tiling by a tile}
\label{thirdtype}
\end{minipage}
\end{figure}

The tiling corresponding to the matrix $\begin{pmatrix}1 & -2 \\ 1 & 0\end{pmatrix}$ and digits $\begin{pmatrix}0 \\ 0\end{pmatrix}$, $\begin{pmatrix}1 \\ 0\end{pmatrix}$ is shown in Figure \ref{thirdtype}. A typical tile has a fractal structure. However, there are some exceptions. The tile corresponding to the matrix $\begin{pmatrix}0 & -2 \\ 1 & 0\end{pmatrix}$ and digits $\begin{pmatrix}0 \\ 0\end{pmatrix}$, $\begin{pmatrix}1 \\ 0\end{pmatrix}$ is a simple rectangle. In Fig. \ref{rect2dim} we can see its partition into its dilated copies. 

\begin{figure}[ht]
\centering
\includegraphics[width=0.4\textwidth]{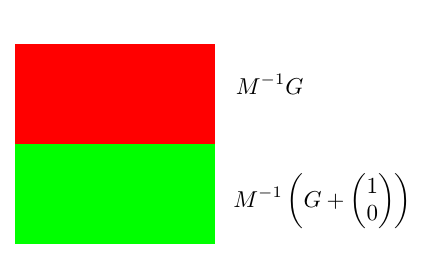}
\caption{A partition of tile-rectangle into two copies similar to it}
\label{rect2dim}
\end{figure}

In this work we classify simple tiles. In the beginning we consider the simplest case of tiles which are parallelepipeds. These tiles are characterized in \cite{GrMad}. We obtain their full classification in Section \ref{theorbox}. Then we study tiles-polyhedra. For two-dimensional plane ($d = 2$) we show that there does not exist polygonal tile (not necesseraly convex), except for parallelepipeds. For arbitrary dimension analogous results stay as an open problem, which is formulated as a conjecture in Section \ref{hyp}. All convex tiles are characterized in Section \ref{convexcase}. Finally, the question about disconnected simple tiles and attractors (which consist of finite number of polyhedra) is, apparently, the most difficult. We give non-trivial examples of such sets in arbitrary dimensions and only in the case $d = 1$ we obtain their complete classification mentioning other known results about integer tiles. 

Obtained results can be used for constructing Haar systems in $\R^d$, so called box-Haar systems which have simple structure and have larger smoothness in com\-pa\-rison to other Haar bases (see Section \ref{boxxhaar}). 

The structure of this work is the following: 
in Section \ref{alldefin} we give the main definitions and recall basic facts about tiles, in Section \ref{multi} we prove a range of results about reducibility, and in Section \ref{theorbox} classify all attractor-parallelepipeds and analyse some examples. In Section \ref{boxxhaar} box-Haar systems are constructed, Section \ref{plane} is devoted to attractors which are polygons in the plane. Then (Section \ref{convexcase}) we classify all convex attractors in $\R^d$, finally, the question about the classification of disconnected attractors is studied in Sections \ref{lineres} and \ref{examptwo}.

\section{Main definitions}
\label{alldefin}
Let $M \in \Z^{d\times d}$ be an integer expanding matrix (all of whose eigenvalues are larger than one in absolute value). Let  $m = |\det{M}|$. The factor-group $\Z^d / M\Z^d$ consists of $m$ equivalent classes. We choose one representative $d_i \in \Z^d$ from each equivalence class, we call this obtained set a \textit{set of digits}: $D(M) = \left\{d_i : i = 0, \ldots, m - 1\right\}$. We always suppose that $0 \in D(M)$. In one-dimensional case, $D(M)$ denotes the set of digits of the $m$-adic system. We always suppose that the matrix and the set of digits are defined and that the set of digits corresponds to this matrix. Let us notice that for every matrix, there are infinitely many corresponding sets of digits.

Consider the following set 

$$G = \left\{\sum \limits_{k = 1}^{\infty}M^{-k}d_{n_k} : d_{n_k} \in D(M)\right\}.$$

In the works \cite{GrHa}, \cite{GrMad}, it is shown that for every expanding integer matrix $M$ and for an arbitrary set of digits $D(M)$, the set $G$ is a compact set with a nonempty interior and possesses the following properties: 

\begin{enumerate}
\item the Lebesgue measure $\mu(G) \in \N$;
\item $G = \bigcup_{d \in D(M)} {M^{-1}(G + d)}$, the sets $M^{-1}(G + d)$ have intersections of zero measure;
\item the indicator function $\chi = \chi_G(x)$ of $G$ almost everywhere satisfies refinement equation  
$$\chi(x) = {\sum \limits_{d \in D(M)}{\chi(Mx - d)}}, \quad x \in \R^d;$$
\item $\sum_{k \in \Z^d}\chi(x + k) \equiv \mu(G)$ almost everywhere, i.e., integer shifts of $\chi$ cover $\R^d$ with $\mu(G)$ layers;
\item $\mu(G) = 1$ if and only if the function system $\{\chi(\cdot + k)\}_{k \in \Z^d}$ is orthonormal.
\end{enumerate}

This allows us to give the following definition.  

\begin{definition}\label{def1}
Let $M \in \Z^{d\times d}$ be the fixed expanding integer matrix and $D(M) = \left\{d_i : i = 0, \ldots, m - 1\right\}$ be the corresponding set of digits. If the Lebesgue measure of the set $G = \left\{\sum_{k = 1}^{\infty}M^{-k}d_{n_k} : d_{n_k} \in D(M)\right\}$ is equal to one (i.e. $\R^d$ is covered by integer shifts of set $G$ with one layer), then the set $G$ is called \textit{tile}.  
\end{definition}

However, in many cases we impose some weaker assumptions on the set $G$. We require the existence of such an expanding matrix $M \in \R^{d \times d}$ (not necessarily integer) and such set $S(M)$ of arbitrary vectors from $\R^d$, that the property $2$ is satisfied. That is: 
\begin{definition}\label{def2}
If for a nonempty compact set $G$, there exists expanding matrix $M \in \R^{d\times d}$ and the set of arbitrary vectors $S(M)$ such that $G = \bigcup_{s \in S(M)} {M^{-1}(G + s)}$ where the sets $M^{-1}(G + s)$ have intersections of zero measure, then the set $G$ is called \textit{attractor}. 
\end{definition}

\begin{remark}
For every set of vectors $S(M)$ and an arbitrary expanding matrix $M$, there exists a unique attractor, because by classical Hutchinson theorem (see \cite{Hutch}) for each finite set of contractions, there is a unique invariant set. 
\end{remark}

When it is possible, we formulate theorems in general case, for attractors. Since tile is a special case of an attractor, it is important to keep in mind the correctness of those results for tiles.

\section{Reducible tiles and attractors}
\label{multi}
We start with the definitions of tensor products of attractors and establish their main properties. 
For arbitrary vectors $a = (a_1, \ldots, a_p)$ and $b = (b_1, \ldots, b_q)$ we denote $a \times b = (a_1, \ldots, a_p, b_1, \ldots, b_q)$. For arbitrary sets of vectors $A$ and $B$ we denote by $A \times B = \{a \times b \mid a \in A, b\in B\}$. 

\begin{definition}
Let us fix an attractor $G_1$ with a matrix $M_1$ and a set of translations $S_1$ in $\R^{d_1}$ and an attractor $G_2$ with a matrix $M_2$ and a set of translations $S_2$ in $\R^{d_2}$. Then we call the set $G_1 \times G_2$ \textit{tensor product of attractors} $G_1$ and $G_2$. 
\end{definition}

\begin{proposition}\label{prop1}
1) The set $G_1 \times G_2 \subset \R^{d_1 + d_2}$ is also an attractor for the block-diagonal matrix $M$, which consists of blocks $M_1$ and $M_2$ and the set of translations $S_1 \times S_2$. 

2) If both sets $G_1$ and $G_2$ are tiles, then their product (as attractors) is also a tile.
\end{proposition}

\begin{proof}
We denote by $\begin{pmatrix}s_1 \\ s_2\end{pmatrix}$ the product $s_1 \times s_2$. To prove 1) we check that 
\begin{eqnarray*}
\bigcup \limits_{s = s_1 \times s_2 \in S_1 \times S_2} {M^{-1}(G + s)} = \bigcup \limits_{\substack{s_1 \times s_2 \in S_1 \times S_2 \\ g_1 \times g_2 \in G_1 \times G_2}}{\begin{pmatrix}{M_1}^{-1} & 0\\ 0 & {M_2}^{-1}\end{pmatrix}\left(\begin{pmatrix}g_1\\ g_2\end{pmatrix} + \begin{pmatrix}s_1\\ s_2\end{pmatrix}\right)} \\ = \bigcup \limits_{s_1 \times s_2 \in S_1 \times S_2} {\begin{pmatrix}M_1^{-1}(G_1 + s_1) \\ M_2^{-1}(G_2 + s_2)\end{pmatrix}} = G_1 \times G_2,
\end{eqnarray*}
and similarly that the sets $$\begin{pmatrix}{M_1}^{-1} & 0\\ 0 & {M_2}^{-1}\end{pmatrix}\left(\begin{pmatrix}{G_1}\\ {G_2}\end{pmatrix} + \begin{pmatrix}s_1\\ s_2\end{pmatrix}\right) \text{ and } \begin{pmatrix}{M_1}^{-1} & 0\\ 0 & {M_2}^{-1}\end{pmatrix}\left(\begin{pmatrix}{G_1}\\ {G_2}\end{pmatrix} + \begin{pmatrix}s_1'\\ s_2'\end{pmatrix}\right)$$ have intersection of measure zero. 

Now we verify 2). Let $G_1$ and $G_2$ be tiles, $D_1 := S_1, D_2 := S_2$ be the corresponding sets of digits. If we consider the set $D_1 \times D_2$ as the set of shifts, we obtain the required set: 
\begin{gather*}G = \left\{\sum \limits_{i = 1}^{\infty}M^{-i}\begin{pmatrix}d_{i_1} \\ d_{i_2}\end{pmatrix}: d_{i_1} \in D_1, d_{i_2} \in D_2 \right\} \\ = \left\{\sum \limits_{i = 1}^{\infty}\begin{pmatrix}M_1^{-i}d_{i_1} \\ M_2^{-i} d_{i_2}  \end{pmatrix}: d_{i_1} \in D_1,  d_{i_2} \in D_2 \right\} \\ = \left\{\sum \limits_{i = 1}^{\infty}M_1^{-i}d_{i_1} : d_{i_1} \in D_1 \right\} \times \left\{\sum \limits_{i = 1}^{\infty}M_2^{-i}d_{i_2} : d_{i_2} \in D_2 \right\} = G_1 \times G_2.\end{gather*} 
It is easy to notice that if integer shifts of $G_1$ cover $\R^{d_1}$ with one layer and integer shifts of $G_2$ cover $\R^{d_2}$ with one layer, then integer shifts of $G_1 \times G_2$ cover $\R^{d_1 + d_2}$ with one layer: 
$$\bigcup \limits_{(z_1, z_2) \in \Z^{d_1 + d_2}}{\left(G_1 \times G_2 + \begin{pmatrix}z_1\\ z_2\end{pmatrix}\right)} = \bigcup \limits_{z_1 \in \Z^{d_1}}{(G_1 + z_1)} \times \bigcup \limits_{z_2 \in \Z^{d_2}}{(G_2 + z_2)} = \R^{{d_1} + {d_2}}.$$ 
If $\begin{pmatrix}z_1\\ z_2\end{pmatrix} \ne \begin{pmatrix}z_3\\ z_4\end{pmatrix}$, we have 
\begin{gather*}\mu\left(\left(G_1 \times G_2 + \begin{pmatrix}z_1\\ z_2\end{pmatrix}\right) \cap \left(G_1 \times G_2 + \begin{pmatrix}z_3\\ z_4\end{pmatrix}\right)\right) \\ =\mu\left(\left(G_1 + z_1\right) \cap \left(G_1 + z_3\right)\right) \mu\left((G_2 + z_2) \cap (G_2 + z_4)\right) = 0.\end{gather*}

Let us also verify that the set of shifts $D_1 \times D_2$ is a well-defined set of digits. Let the first shift from $D_1 \times D_2$ consist of $d_{k_1} \in D_1$ and $d_{k_2} \in D_2$, the second shift consist of $d_{l_1} \in D_1$ and $d_{l_2} \in D_2$. 
We should check that they are in different classes of equivalency with respect to the matrix $M$.
If 
$$\left(\begin{pmatrix}d_{k_1} \\ d_{k_2}\end{pmatrix} - \begin{pmatrix}d_{l_1} \\ d_{l_2}\end{pmatrix}\right) = \begin{pmatrix}M_1 & 0 \\ 0 & M_2\end{pmatrix} \begin{pmatrix}z_1 \\ z_2\end{pmatrix},$$ then 
$$(d_{k_1} - d_{l_1}) = M_1 z_1, (d_{k_2} - d_{l_2}) = M_2 z_2 \Rightarrow d_{k_1}= d_{l_1}, d_{k_2} = d_{l_2},$$ i.e. it is possible only if the digits are equal. 

Proposition \ref{prop1} is proved. 
\end{proof}

%\vspace{\baselineskip}

\begin{definition}
An attractor is \textit{irreducible} if it is not a tensor product of other attractors. 
\end{definition}

In what follows we analyse in which cases attractors can be simple sets instead of being fractals: when they are parallelepipeds, polygons, polyhedra. In Section \ref{theorbox} we characterize irreducible attractors which are parallelepipeds and show that all other attractors-parallelepipeds are their tensor product.

\section{The classification of box-attractors}
\label{theorbox}
Let us investigate the structure of ``simple'' attractors that are parallelepipedes (any, not only rectangular). We call such sets \textit{box-attractors}. If the attractor is a tile, we call it a \textit{box-tile}. 

It is clear that every attractor is a product of irreducible attractors. In Theorem \ref{mainbox} we prove that a box-attractor is decomposed into a product of not just attractors but box-attractors (irreducible). Similarly, a box-tile is decomposed into a product of box-tiles. Then in Theorem \ref{tbox2}
 we classify irreducible box-attractors and further give some examples. 

\begin{theorem} \label{mainbox}
Each box-attractor is in a suitable basis a tensor product of  irreducible box-attractors; each box-tile in a suitable basis is a tensor product of irreducible box-tiles. 
\end{theorem}

\begin{proof}
Let the initial attractor $G$ be generated by the matrix $M$ and the set of shifts $S$. Since it is an attractor, it is a union of the non-overlapping copies (shifts) $G_1, \ldots, G_k$ of its contraction obtained with the matrix $M^{-1}$. 

First we establish that a parallelepiped can be tiled (without overlappings) with some shifts of a similar parallelepiped only if the faces of small parallelepipeds are parallel to the faces of the big one. As a consequence, the edges of small parallelepipeds are also parallel to the edges of the big parallelepiped. Indeed, for each face of the big parallelepiped, there is a parallel face of a small parallelepiped (for example, the one that is adjacent to it). Each of the parallelepipeds consists of the pairs of parallel faces. Since parallelepipeds have equal number of faces, for each face of a small parallelepiped there is a parallel face of the big parallelepiped. 

Consider the coordinate system with axes along the edges of the initial pa\-ralle\-le\-pi\-ped $G$ with the origin in one of the vertices. Since the edges of the parallelepiped $M^{-1}G$ are parallel to the edges of the parallelepiped $G$, the directions of the edges of $M^{-1}G$ are rearranged directions of the basis vectors. This permutation is split into the cycles, hence the matrix $M^{-1}$ has in the new basis a block-diagonal form, where each block corresponds to its own cycle. For each cycle, let us consider the projection $G'$ of the initial parallelepiped $G$ onto the subspace $X$ spanned by the basis vectors from this cycle. Denote by $M_1$ the block in the matrix $M$ (in the new basis) corresponding to an arbitrary fixed cycle. Consider those of sets $G_1, \ldots, G_k$ that are adjacent to $G'$, denote them by $G_1, \ldots, G_q$. Let them be obtained from the parallelepiped $G_1$ with shifts $s_1, \ldots, s_q \in S$. The shifts $s_1, \ldots, s_q \in S$ are parallel to $X$ since otherwise they have the part orthogonal to $X$ and $G_1, \ldots, G_q$ are not all adjacent to $X$. The union of sets $G_1', \ldots, G_q'$ (the projections of $G_1, \ldots, G_q$ on the $X$) is the parallepiped $G'$. Indeed, since $G$ is tiled with the sets $G_1, \ldots, G_k$, its projection $G'$ onto $X$ is covered with adjacent to $G'$ sets and these projections do not overlap (otherwise $G_1, \ldots, G_q$ overlap since they are shifts of each other along $X$). All sets $G_i'$ are shifted copies of the set $M_1^{-1}G'$. So $G'$ is the box-attractor since it is the union of its non-overlapping dilated copies $G_i'$ obtained under the action of the  matrix $M_1$. In other words, $G'$ is the box-attractor corresponding to the fixed cycle. 

If the initial attractor $G$ is a tile, then the initial set of shifts is the set of digits for the matrix $M$. All the vectors $s_1, \ldots, s_q$ are parallel to the space $X$, therefore we can keep only their coordinates corresponding to $X$ and obtain integer vectors $s_1', \ldots, s_q'$ of smaller dimension (just without some zeros). These vectors form the set of digits to the matrix $M_1$. Indeed, otherwise if $s'_{k} - s'_{l} = M_1z_1$, then $s_k - s_l = Mz$ (since $s_k$ and $s'_k$ differ only in zeros). We can cover $G'$ with the use of vectors $s_1', \ldots, s_q'$ and the matrix $M_1$, so  
$$G' = \bigcup \limits_{s \in \{s_1', \ldots, s_q'\}} {M_1^{-1}(G' + s)},$$
 and  the sets $M_1^{-1}(G' + s)$ are non-overlapping. Hence  
$$G' = \left\{\sum \limits_{i = 1}^{\infty}M_1^{-i}s_i : s_i \in \{s_1', \ldots, s_q'\}\right\}.$$
 Besides that, integer shifts of $G'$ tile the space $X$ with one layer (as integer shifts of $G$ tile $\R^d$ with one layer, and we keep only shifts parallel to $X$, we obtain the tiling of $X$). Thus, $G'$ satisfies the definition of a tile. 
Therefore, a box-tile is decomposed in box-tiles. 
 
The attractor $G'$ corresponding to one cycle is irreducible. Otherwise the permutation of its basis directions decomposes into  independent cycles corresponding to the multipliers in the tensor product. 

Theorem \ref{mainbox} is proved.  
\end{proof}

\begin{theorem} \label{tbox2}
Suppose $M \in \Z^{d \times d}$ is an expanding matrix. Then the following conditions are equivalent: 

a) There exists an irreducible box-attractor, constructed by this matrix and some set of integer shifts.

b) In suitable basis over $\Z^d$ the matrix has the following form: 
$$\begin{pmatrix}
0 & p_1 & 0 & \ldots & 0\\
0 & 0 & p_2 & \ldots & 0\\
\ldots & \ldots & \ldots & \ldots &\ldots \\
0 & 0 & 0 & \ldots & p_{n - 1} \\
\pm p_n & 0 & 0 & \ldots & 0\\
\end{pmatrix},$$ where all $p_i$ are natural numbers. 
Also, the product of all $p_i$ is equal to the determinant of the matrix $M$ in the absolute value. 
\end{theorem}

\begin{proof}
Let $M$ be an integer expanding matrix.

Let us prove that a) implies b). Consider an irreducible box-attractor $G$ generated by the matrix $M$ and a set of integer shifts $S$. 
The matrix $M^{-1}$ permutes the directions along edges. Hence, there is its power $M^{-k}$ acting as a contraction along the edges. Let us show that the attractor $G$ can be also generated by the matrix $M^k$ with some shifts. Indeed, 
\begin{gather*}G = \bigcup \limits_{s \in \{s_1, \ldots, s_m\}} {M^{-1}(G + s)} = \bigcup \limits_{s \in \{s_1, \ldots, s_m\}} {M^{-1}\left(\bigcup \limits_{s' \in \{s_1, \ldots, s_m\}} {M^{-1}(G + s')} + s\right)} \\= \bigcup \limits_{s, s' \in \{s_1, \ldots, s_m\}} {M^{-2}(G + s' + Ms).}
\end{gather*}
Therefore, we may assume that the generating matrix is the matrix $M^{2}$ and the corresponding set of its shifts is the set $\{G + s' + Ms \mid s, s' \in \{s_1, \ldots, s_m\}\}$. The shifts are integer as the matrix $M$ is integer. Similarly, it holds for bigger powers, in particular, for the matrix $M^k$. 

Now prove that if a box-attractor $G$ is generated by an integer matrix $M_1$ that acts as the expansion along its edges (in our case $M_1 = M^k$) and integer shifts, then its edges are vectors with integer coordinates. The parallelepiped $G$ is split as a lattice (see Fig. \ref{zproof}). Fix an arbitrary edge $e$; let the matrix $M_1^{-1}$ contract it in $h$ times. There are several contracted sets adjacent to the edge $e$, the first one is the set $M_1^{-1}G$. Let the next one be the set $M_1^{-1}(G + s_0)$ (the neighbour along the edge $e$). Due to properties of the matrix $M_1$, the shift $s_0$ are along the edge $e$. If we shift the set $M_1^{-1}G$ by the vector $s_1 = M_1^{-1}s_0 = s_0 / h$, we obtain exactly the set $M_1^{-1}(G + s_0)$. On the other hand, these two sets are adjacent to each other as they are the neighbors among sets intersecting with the edge $e$. Consequently, the edge $e / h$ of the parallelepiped $M_1^{-1}G$ moves by the shift $s_1$ to the edge of the parallelepiped $M_1^{-1}(G + s_0)$. Thus $|e| / h = |s_1| = |s_0| / h$, and then, as $e$ and $s_0$ have the same direction, they are equal, hence, $e$ is an integer vector.  

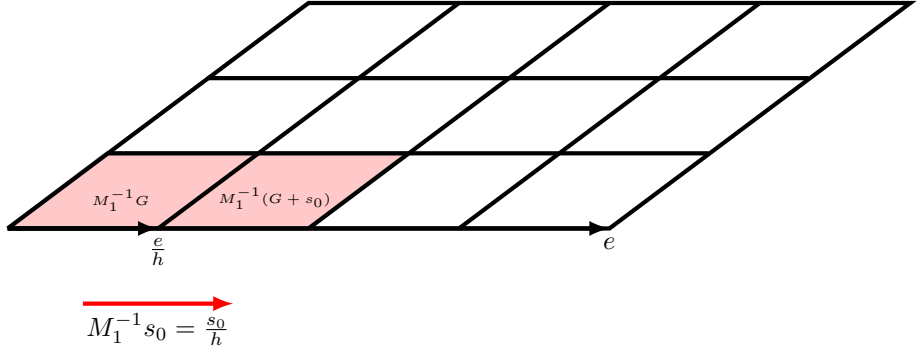
\begin{figure}[ht]
\centering
\begin{tikzpicture}[scale = 0.5]
	\coordinate (X1) at (-8cm, -3cm);
	\coordinate (X2) at (0cm, 3cm);
	\coordinate (X3) at (16cm, 3cm);
	\coordinate (X4) at (8cm, -3cm);

	\draw[ultra thick, black, fill=red!70, opacity = 0.3] (X1) -- ($(-2.68*2, -1)$) -- ($(-2.68*2 + 8, -1)$) -- ($(0, -3)$) -- (X1);

	\draw[ultra thick, black] (X1)--(X2)--(X3)--(X4)--(X1);
	\draw[ultra thick, black, -latex] (X1) -- (X4);

	\draw[ultra thick, black, -latex] (X1) -- ($(-4, -3)$);
	\draw[ultra thick, black] ($(-4, -3)$) -- ($(4, 3)$);
	\draw[ultra thick, black] ($(0, -3)$) -- ($(8, 3)$);
	\draw[ultra thick, black] ($(4, -3)$) -- ($(12, 3)$);

	\draw[ultra thick, black] ($(-2.68*2, -1)$) -- ($(-2.68*2 + 16, -1)$);
	\draw[ultra thick, black] ($(-1.35*2, 1)$) -- ($(16 - 2*1.35, 1)$);
	\node at (X4)[below] {$e$};
	\node at (barycentric cs:X1=7,X3=1) {\tiny{$M_1^{-1}G$}};
\node at ($(-4, -3)$)[below] {$\frac{e}{h}$};
	\node at (-0.9, -2.2) {\tiny{$M_1^{-1}(G + s_0)$}};
	\draw[ultra thick, red, -latex] ($(-6, -5)$)--($(-2, -5)$);
	\node at (-4, -5)[below] {$M_1^{-1}s_0 = \frac{s_0}{h}$};
	\end{tikzpicture}
\caption{The proof that the edges of $G$ are integer.}
\label{zproof}
\end{figure}

Fix a vertex of the parallelellepiped $G$ and consider its adjacent edges. Consider the integer basis consisting of those integer vectors. The matrix $M^{-1}$ acts as a cycling permutation of this basis. We renumber and change the directions of basis vectors to assume that the image of the edge $e_1$ under the action of $M^{-1}$ is $e_2 \cdot k_1$, the image of $e_2$ is $e_3 \cdot k_2$, etc., the image of $e_n$ is $\pm e_1 \cdot k_n$ (the new basis is also integer). Since the initial parallelepiped is tiled with contracted parallelepipeds, we obtain $k_i = 1/p_i$ where $p_i$ is a natural number (it is the number of parts obtained by splitting the $i$-th edge). In total, there are $p_1 \cdots p_n$ parts, therefore this product is equal to the determinant of the initial matrix in the absolute value (and it is the number of digits in case of a tile). Then the matrix $M^{-1}$ in this basis has the following form: 
 $$\begin{pmatrix}
0 & 0 & \ldots & 0 & \pm \frac{1}{p_n} \\
\frac{1}{p_1} & 0 & 0 & \ldots & 0 \\
0 & \frac{1}{p_2} & 0 & \ldots & 0 \\
\ldots & \ldots & \ldots & \ldots & \ldots \\
0 & 0 & \ldots & \frac{1}{p_{n - 1}} & 0 \\
\end{pmatrix}.$$ 
Thus, the matrix $M$ in this basis has the form  
$$\begin{pmatrix}
0 & p_1 & 0 & \ldots & 0\\
0 & 0 & p_2 & \ldots & 0\\
\ldots & \ldots & \ldots & \ldots &\ldots \\
0 & 0 & 0 & \ldots & p_{n - 1} \\
\pm p_n & 0 & 0 & \ldots & 0\\
\end{pmatrix},$$ as it was required to prove for the necessity.  

Now show that b) implies a). From the form of matrix (in some basis) it follows that under the action of contraction $M^{-1}$ the basis directions permute in cycle. Hence, any attractor generated by $M$ is irreducible (otherwise the generated permutation has several cycles). We construct the parallelepiped $G$ spanned on the vectors of our integer basis. Due to the form of the matrix $M$, each edge of the parallelepiped $G$ is split under the action of the contraction $M^{-1}$ into the integer number of parts (depending on $p_i$). Hence, we can tile the whole set $G$ with shifts of the set $M^{-1}G$. The shortest vector of the shifts along the edge $e$ of the parallelepiped $G$ is $e$ divided by some $N$, where $N$ is an integer. The vector $e$ is integer, therefore in the  initial basis all vectors of shifts along the edge $e$ have rational coordinates (these vectors have the form $k_e^j e / N_e$). We multiply all shift vectors (for all edges $e$) simultaneously by a big natural number $N$ so that all new shift vectors are integer. The parallelepiped $G$ also expands to the parallelepiped $G_1$. Then the attractor $G_1$ generated by these expanded integer shifts (which have the form $Nk_e^j e / N_e$) and the matrix $M$ is also an irreducible box-attractor that is required.

Theorem \ref{tbox2} is proved. 

\begin{remark} \label{rem1}
Let us fix an arbitrary irreducible box-attractor in its cyclic form from b). Then note that not all $p_i = 1$. Otherwise we have $M^qe_1 = e_1$, where $q$ is the length of the cycle; it is a contradiction with the dilation property of the matrix $M^{-1}$. 
\end{remark}
\end{proof}

\begin{corollary}
A box-attractor generated by a dilation matrix with prime determinant is irreducible. 
\end{corollary}

Indeed, otherwise the matrix is decomposed into the block-diagonal parts corresponding to irreducible attractors, then the determinant is a product of the determinants that are not all equal to one because of Remark \ref{rem1}. 

Let us illustrate the theorems with several examples. 

\begin{example}
Consider the matrix $M = \begin{pmatrix}0 & -2 \\ 1 & 0\end{pmatrix}$ with the determinant $2$ which we have already met in Section \ref{intro}. It consists of one cyclic block and generates a box-tile, for example, with digits $\begin{pmatrix}0 \\ 0\end{pmatrix}$, $\begin{pmatrix}1 \\ 0\end{pmatrix}$; see Fig. \ref{rect2dim}. 
\end{example}

\begin{example}
{
Let us construct an irreducible box-attractor in $\R^3$. We use the matrix 
$$M = \begin{pmatrix}
0 & 1 & 0\\
0 & 0 & 1\\
2 & 0 & 0\\
\end{pmatrix}.$$ 
Since $|\det{M}| = 2$, there are two shifts, so the box-attractor is divided with the contraction $M^{-1}$ into two parts. To analyse how it works, consider the matrix  
$$M^{-1} = \begin{pmatrix}
0 & 0 & -\frac{1}{2}\\
1 & 0 & 0\\
0 & 1 & 0 \\
\end{pmatrix}.$$
The partition is represented in Fig. \ref{box1}.
}
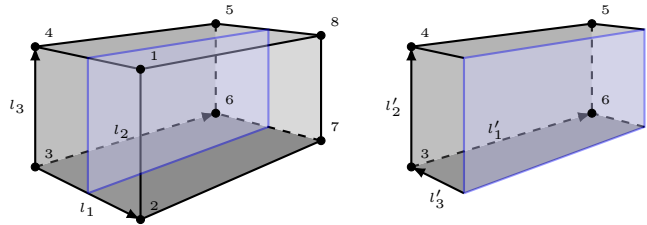
\begin{figure}[ht]
  \centering
{
\begin{tikzpicture}
	\coordinate (P1) at (-7cm,1.5cm); 
	\coordinate (P2) at (8cm,1.5cm); 
	\coordinate (A1) at (0em,0cm);
	\coordinate (A2) at (0em,-2cm);
	\coordinate (A3) at ($(P1)!.8!(A2)$); 
	\coordinate (A4) at ($(P1)!.8!(A1)$);
	\coordinate (A7) at ($(P2)!.7!(A2)$);
	\coordinate (A8) at ($(P2)!.7!(A1)$);
	\coordinate (A5) at
	  (intersection cs: first line={(A8) -- (P1)},
			    second line={(A4) -- (P2)});
	\coordinate (A6) at
	  (intersection cs: first line={(A7) -- (P1)}, 
			    second line={(A3) -- (P2)});
	\fill[gray!90] (A2) -- (A3) -- (A6) -- (A7) -- cycle; % face 6
	\fill[gray!50] (A3) -- (A4) -- (A5) -- (A6) -- cycle; % face 3	
	\fill[gray!30] (A5) -- (A6) -- (A7) -- (A8) -- cycle; % face 4
	\draw[thick,dashed] (A5) -- (A6);
	\draw[thick,dashed, -latex] (A3) -- (A6);
	\draw[thick,dashed] (A7) -- (A6);
	\fill[gray!50,opacity=0.2] (A1) -- (A2) -- (A3) -- (A4) -- cycle; % f2
	\fill[gray!90,opacity=0.2] (A1) -- (A4) -- (A5) -- (A8) -- cycle; % f5
	\draw[thick] (A1) -- (A2);
	\draw[thick, -latex] (A3) -- (A4);
	\draw[thick] (A7) -- (A8);
	\draw[thick] (A1) -- (A4);
	\draw[thick] (A1) -- (A8);
	\draw[thick, -latex] (A3) -- (A2);
	\draw[thick] (A2) -- (A7);
	\draw[thick] (A4) -- (A5);
	\draw[thick] (A8) -- (A5);
	\coordinate (W32) at (barycentric cs:A2=1,A3=1);
	\coordinate (W41) at (barycentric cs:A4=1,A1=1);
	\coordinate (W58) at (barycentric cs:A5=1,A8=1);
	\coordinate (W67) at (barycentric cs:A6=1,A7=1);
	\draw[blue, thick, fill = blue!20, opacity=0.4] (W32) -- (W41) -- (W58) -- (W67) -- (W32);
	\node [below] at (barycentric cs:A2=1,A3=1) {\tiny $l_1$};
	\node [left] at (barycentric cs:A3=1,A4=1) {\tiny $l_3$};
	\node [left] at (barycentric cs:A3=1,A6=1.5,A4=0.1) {\tiny $l_2$};
	\foreach \i in {1,2,...,8}
	{
	  \draw[fill=black] (A\i) circle (0.15em)
	    node[above right] {\tiny \i};
	}
	\coordinate (PP1) at (-2cm,1.5cm);
	\coordinate (PP2) at (13cm,1.5cm);
	\coordinate (PA1) at (5cm,0cm);
	\coordinate (PA2) at (5cm,-2cm);
	\coordinate (PA3) at ($(PP1)!.8!(PA2)$); 
	\coordinate (PA4) at ($(PP1)!.8!(PA1)$);
	\coordinate (PA7) at ($(PP2)!.7!(PA2)$);
	\coordinate (PA8) at ($(PP2)!.7!(PA1)$);
	\coordinate (PA5) at
	  (intersection cs: first line={(PA8) -- (PP1)},
			    second line={(PA4) -- (PP2)});
	\coordinate (PA6) at
	  (intersection cs: first line={(PA7) -- (PP1)}, 
			    second line={(PA3) -- (PP2)});
	\coordinate (W32) at (barycentric cs:PA2=1,PA3=1);
	\coordinate (W41) at (barycentric cs:PA4=1,PA1=1);
	\coordinate (W58) at (barycentric cs:PA5=1,PA8=1);
	\coordinate (W67) at (barycentric cs:PA6=1,PA7=1);
	\fill[gray!90] (W32) -- (PA3) -- (PA6) -- (W67) -- cycle; % face 6
	\fill[gray!50] (PA3) -- (PA4) -- (PA5) -- (PA6) -- cycle; % face 3
	\fill[gray!30] (PA5) -- (PA6) -- (W67) -- (W58) -- cycle; % face 4
	\draw[thick,dashed] (PA5) -- (PA6);
	\draw[thick,dashed, -latex] (PA3) -- (PA6);
	\draw[thick,dashed] (W67) -- (PA6);
	\fill[gray!50,opacity=0.2] (W41) -- (W32) -- (PA3) -- (PA4) -- cycle; % f2
	\fill[gray!90,opacity=0.2] (W41) -- (PA4) -- (PA5) -- (W58) -- cycle; % f5
	\draw[thick, -latex] (PA3) -- (PA4);
	\draw[thick] (W41) -- (PA4);
	\draw[thick, -latex] (W32) -- (PA3);
	\draw[thick] (PA4) -- (PA5);
	\draw[thick] (W58) -- (PA5);	
	\draw[blue, thick, fill = blue!20, opacity=0.4] (W32) -- (W41) -- (W58) -- (W67) -- (W32);
	\node [below] at (barycentric cs:PA3=1,W32=1) {\tiny $l_3'$};
	\node [left] at (barycentric cs:PA3=1,PA4=1) {\tiny $l_2'$};
	\node [left] at (barycentric cs:PA3=1,PA6=1.5,PA4=0.15) {\tiny $l_1'$};
	\foreach \i in {4, 5, 6, 3}
	{
	  \draw[fill=black] (PA\i) circle (0.15em)
	    node[above right] {\tiny \i};
	}
\end{tikzpicture}
\caption{The example of a box-attractor with two shifts.}
\label{box1}
}
\end{figure}
\end{example}

\begin{example}
{
The matrix 
$$M = \begin{pmatrix}
0 & 3 & 0\\
0 & 0 & 2\\
2 & 0 & 0\\
\end{pmatrix}, \quad \text{where } M^{-1}=  
\begin{pmatrix}
0 & 0 & \frac{1}{2}\\
\frac{1}{3} & 0 & 0\\
0 & \frac{1}{2} & 0 \\
\end{pmatrix},$$
generates an irreducible box-attractor with $12$ shifts. 
The partition is depicted in Fig. \ref{box2}.
}
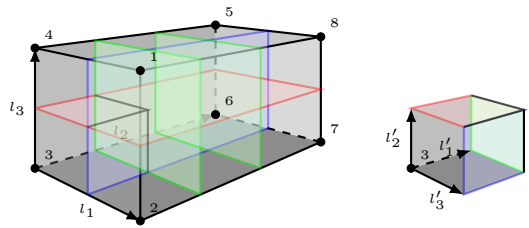
\begin{figure}[ht]
\centering
 {

\begin{tikzpicture}
	\coordinate (P1) at (-7cm,1.5cm);
	\coordinate (P2) at (8cm,1.5cm);
	\coordinate (A1) at (0em,0cm); 
	\coordinate (A2) at (0em,-2cm); 
	\coordinate (A3) at ($(P1)!.8!(A2)$);
	\coordinate (A4) at ($(P1)!.8!(A1)$);
	\coordinate (A7) at ($(P2)!.7!(A2)$);
	\coordinate (A8) at ($(P2)!.7!(A1)$);
	\coordinate (A5) at
	  (intersection cs: first line={(A8) -- (P1)},
			    second line={(A4) -- (P2)});
	\coordinate (A6) at
	  (intersection cs: first line={(A7) -- (P1)}, 
			    second line={(A3) -- (P2)});
	\fill[gray!90] (A2) -- (A3) -- (A6) -- (A7) -- cycle; % face 6
	\fill[gray!50] (A3) -- (A4) -- (A5) -- (A6) -- cycle; % face 3	
	\fill[gray!30] (A5) -- (A6) -- (A7) -- (A8) -- cycle; % face 4
	\draw[thick,dashed] (A5) -- (A6);
	\draw[thick,dashed, -latex] (A3) -- (A6);
	\draw[thick,dashed] (A7) -- (A6);
	\fill[gray!50,opacity=0.2] (A1) -- (A2) -- (A3) -- (A4) -- cycle; % f2
	\fill[gray!90,opacity=0.2] (A1) -- (A4) -- (A5) -- (A8) -- cycle; % f5
	\draw[thick] (A1) -- (A2);
	\draw[thick, -latex] (A3) -- (A4);
	\draw[thick] (A7) -- (A8);
	\draw[thick] (A1) -- (A4);
	\draw[thick] (A1) -- (A8);
	\draw[thick, -latex] (A3) -- (A2);
	\draw[thick] (A2) -- (A7);
	\draw[thick] (A4) -- (A5);
	\draw[thick] (A8) -- (A5);
	\coordinate (W32) at (barycentric cs:A2=1,A3=1);
	\coordinate (W41) at (barycentric cs:A4=1,A1=1);
	\coordinate (W58) at (barycentric cs:A5=1,A8=1);
	\coordinate (W67) at (barycentric cs:A6=1,A7=1);
	\draw[blue, thick, fill = blue!20, opacity=0.4] (W32) -- (W41) -- (W58) -- (W67) -- (W32);
	\node [below] at (barycentric cs:A2=1,A3=1) {\tiny $l_1$};
	\node [left] at (barycentric cs:A3=1,A4=1) {\tiny $l_3$};
	\node [left] at (barycentric cs:A3=1,A6=1.5,A4=0.1) {\tiny $l_2$};
\coordinate (W34) at (barycentric cs:A3=1,A4=1);
	\coordinate (W12) at (barycentric cs:A1=1,A2=1);
	\coordinate (W78) at (barycentric cs:A7=1,A8=1);
	\coordinate (W56) at (barycentric cs:A5=1,A6=1);
\draw[red, thick, fill = red!20, opacity=0.4] (W34) -- (W12) -- (W78) -- (W56) -- (W34);

\coordinate (WF45) at (barycentric cs:A4=2,A5=1);
	\coordinate (WF18) at (barycentric cs:A1=2,A8=1);
	\coordinate (WF27) at (barycentric cs:A2=2,A7=1);
	\coordinate (WF36) at (barycentric cs:A3=2,A6=1);
\draw[green, thick, fill = green!20, opacity=0.4] (WF45) -- (WF18) -- (WF27) -- (WF36) -- (WF45);

\coordinate (WB45) at (barycentric cs:A4=1,A5=2);
	\coordinate (WB18) at (barycentric cs:A1=1,A8=2);
	\coordinate (WB27) at (barycentric cs:A2=1,A7=2);
	\coordinate (WB36) at (barycentric cs:A3=1,A6=2);
\draw[green, thick, fill = green!20, opacity=0.1] (WF45) -- (WF18) -- (WF27) -- (WF36) -- (WF45);
\draw[green, thick, fill = green!20, opacity=0.4] (WB45) -- (WB18) -- (WB27) -- (WB36) -- (WB45);

\coordinate (Wmy1) at (barycentric cs:W34=1,W12=1);
\coordinate (W68) at (barycentric cs:A6=1,A8=1);
\coordinate (Wmy2) at (barycentric cs:Wmy1=2,W68=1);
\coordinate (Wmy3) at (barycentric cs:W34=2,W56=1);
\coordinate (Wmy4) at (barycentric cs:W32=2,W67=1);

\draw[black, thick, opacity=0.4] (Wmy1) -- (Wmy2) -- (Wmy3);
\draw[black, thick, opacity=0.4] (Wmy2) -- (Wmy4);

	\foreach \i in {1,2,...,8}
	{
	  \draw[fill=black] (A\i) circle (0.15em)
	    node[above right] {\tiny \i};
	}
	\coordinate (PP1) at (-2cm,1.5cm);
	\coordinate (PP2) at (13cm,1.5cm);
	\coordinate (PA1) at (5cm,0cm);
	\coordinate (PA2) at (5cm,-2cm);
	\coordinate (PA3) at ($(PP1)!.8!(PA2)$); 
	\coordinate (PA4) at ($(PP1)!.8!(PA1)$);
	\coordinate (PA7) at ($(PP2)!.7!(PA2)$);
	\coordinate (PA8) at ($(PP2)!.7!(PA1)$);
	\coordinate (PA5) at
	  (intersection cs: first line={(PA8) -- (PP1)},
			    second line={(PA4) -- (PP2)});
	\coordinate (PA6) at
	  (intersection cs: first line={(PA7) -- (PP1)}, 
			    second line={(PA3) -- (PP2)});
	\coordinate (W32) at (barycentric cs:PA2=1,PA3=1);
	\coordinate (W41) at (barycentric cs:PA4=1,PA1=1);
	\coordinate (W58) at (barycentric cs:PA5=1,PA8=1);
	\coordinate (W67) at (barycentric cs:PA6=1,PA7=1);
	
	\coordinate (W34) at (barycentric cs:PA3=1,PA4=1);
	\coordinate (W12) at (barycentric cs:PA1=1,PA2=1);
	\coordinate (W78) at (barycentric cs:PA7=1,PA8=1);
	\coordinate (W56) at (barycentric cs:PA5=1,PA6=1);

\coordinate (WF45) at (barycentric cs:PA4=2,PA5=1);
	\coordinate (WF18) at (barycentric cs:PA1=2,PA8=1);
	\coordinate (WF27) at (barycentric cs:PA2=2,PA7=1);
	\coordinate (WF36) at (barycentric cs:PA3=2,PA6=1);

\coordinate (WB45) at (barycentric cs:PA4=1,PA5=2);
	\coordinate (WB18) at (barycentric cs:PA1=1,PA8=2);
	\coordinate (WB27) at (barycentric cs:PA2=1,PA7=2);
	\coordinate (WB36) at (barycentric cs:PA3=1,PA6=2);

\coordinate (Wmy1) at (barycentric cs:W34=1,W12=1);
\coordinate (W68) at (barycentric cs:PA6=1,PA8=1);
\coordinate (Wmy2) at (barycentric cs:Wmy1=2,W68=1);
\coordinate (Wmy3) at (barycentric cs:W34=2,W56=1);
\coordinate (Wmy4) at (barycentric cs:W32=2,W67=1);
\coordinate (Wmy5) at (barycentric cs:PA3=2,PA6=1);
\fill[gray!90] (W32) -- (PA3) -- (Wmy5) -- (Wmy4) -- cycle; % face 6
\fill[gray!50] (PA3) -- (W34) -- (Wmy3) -- (Wmy5) -- cycle; % face 3
\draw[blue, thick, fill = blue!20, opacity=0.4] (Wmy4) -- (W32) -- (Wmy1)--(Wmy2)--(Wmy4);
	\fill[red!20, opacity = 0.4] (Wmy1) -- (W34) -- (Wmy3) -- (Wmy2) -- cycle;
\fill[green!20, opacity = 0.4] (Wmy3) -- (Wmy5) -- (Wmy4) -- (Wmy2) -- cycle;
	
\draw[red, thick, opacity=0.4] (Wmy1) -- (W34) -- (Wmy3);
\draw[green, thick, opacity=0.4] (Wmy3) -- (Wmy5) -- (Wmy4);

\draw[thick, -latex] (PA3) -- (W34);
	\draw[thick, -latex] (PA3) -- (W32);
	\draw[black, thick, opacity=0.7] (Wmy1) -- (Wmy2) -- (Wmy3);

\draw[black, thick, opacity=1] (Wmy2) -- (Wmy4);
\draw[thick,dashed, -latex] (PA3) -- (Wmy5);
\node [below] at (barycentric cs:PA3=1,W32=1) {\tiny $l_3'$};
	\node [left] at (barycentric cs:PA3=1,W34=1) {\tiny $l_2'$};
	\node at (barycentric cs:PA3=1,Wmy5=1,Wmy3=0.5) {\tiny $l_1'$};

	\foreach \i in {3}
	{
	  \draw[fill=black] (PA\i) circle (0.15em)
	    node[above right] {\tiny \i};
	}
\end{tikzpicture}
}
\caption{The example of a box-attractor with twelve shifts.}
\label{box2}
\end{figure}
\end{example}

\section{Box-Haar systems}
\label{boxxhaar}
In this section we recall the definition of the Haar bases and spot the advantages of the Haar systems constructed by the box-tiles described in Section \ref{theorbox}. 

Using an arbitrary tile in $\R^d$ it is possible to construct the basis in the space $L_2(\R^d)$. Let us consider in details this construction described, for example, in \cite{GrMad}. 

Recall the construction of the classical one-dimensional Haar system on the segment $[0, 1]$. It is performed in several levels. First we define the function 
$$h_0(x) = \chi[0, 1).$$
Then the functions of level $j$ are constructed by the general formula 
\begin{equation*}
h_{2^j + k}(t) = 
 \begin{cases}
	(\sqrt{2})^{j}, & t \in [\frac{k}{2^j}, \frac{k + \frac{1}{2}}{2^{j}}) \\

    -(\sqrt{2})^{j}, & t \in [\frac{k + \frac{1}{2}}{2^{j}}, \frac{k + 1}{2^j}), \quad \quad \quad k = 0, \ldots, 2^{j} - 1. \\

	0 & otherwise, \\
 \end{cases}
\end{equation*}

The union of $h_0(x)$ and the functions of all levels forms the basis in $L_2[0, 1]$. The first few basis functions are shown in Fig. \ref{haar0} -- \ref{haar3}. 

\begin{figure}[ht]
\begin{minipage}{0.45\linewidth}
\begin{tikzpicture}
\draw[thick, black, -latex] 
(0, 5) -- ($(5, 5)$);
\draw [thick, black, -latex] 
(0, 2) -- ($(0, 9)$);
\node at (5, 5) [below]{$x$};
\node at (0, 9) [left]{$y$};
\draw[ultra thick, red] 
(0, 7.5) -- ($(2.5, 7.5)$);
\draw[dashed, thick, red] 
(0, 7.5) -- ($(0, 5)$);
\draw[dashed, thick, red] 
(2.5, 5) -- ($(2.5, 7.5)$);

\node at (0, 5) [below right] {$0$};
\node at (2.5, 5) [below right] {$1$};
\node at (0, 7.5) [left] {$1$};
\node at (0, 2.5) [left] {$-1$};
\node at (0.3, 8.3) [right] {\large $h_0 = \chi_{[0,  1]}$};
\end{tikzpicture}
\caption{The function $h_0$ that is beyond levels of Haar system.}
\label{haar0}
\end{minipage}
\begin{minipage}{0.08\linewidth}
\quad \quad
\end{minipage}
\begin{minipage}{0.45\linewidth}
\begin{tikzpicture}
\draw[thick, black, -latex] 
(0, 5) -- ($(5, 5)$);
\draw [thick, black, -latex] 
(0, 2) -- ($(0, 9)$);
\node at (5, 5) [below]{$x$};
\node at (0, 9) [left]{$y$};
\draw[ultra thick, red] 
(0, 7.5) -- ($(1.25, 7.5)$);
\draw[dashed, thick, red] 
(0, 7.5) -- ($(0, 5)$);
\draw[dashed, thick, red] 
(1.25, 5) -- ($(1.25, 7.5)$);

\draw[ultra thick, red] 
(2.5, 2.5) -- ($(1.25, 2.5)$);
\draw[dashed, thick, red] 
(2.5, 2.5) -- ($(2.5, 5)$);
\draw[dashed, thick, red] 
(1.25, 5) -- ($(1.25, 2.5)$);

\node at (0, 5) [below right] {$0$};

\node at (1.25, 5) [below right] {$0.5$};

\node at (2.5, 5) [below right] {$1$};

\node at (0, 7.5) [left] {$1$};

\node at (0, 2.5) [left] {$-1$};

\node at (0.3, 8.3) [right] {\large $h_1 = \chi_{[0,  \frac{1}{2}]} - \chi_{[\frac{1}{2}, 1]}$};
\end{tikzpicture}
\caption{The function $h_1$ that is in the zero level of Haar system.}
\label{haar1}
\end{minipage}

\end{figure}

\begin{figure}[ht]
\begin{minipage}{0.45\linewidth}
\begin{tikzpicture}[scale = 0.8]
\draw[thick, black, -latex] 
(0, 5) -- ($(5.3, 5)$);
\draw [thick, black, -latex] 
(0, 2) -- ($(0, 9)$);
\node at (5.3, 5) [below]{$x$};
\node at (0, 9) [left]{$y$};
\draw[ultra thick, red] 
(0, 7.5) -- ($(1.25, 7.5)$);
\draw[dashed, thick, red] 
(0, 7.5) -- ($(0, 5)$);
\draw[dashed, thick, red] 
(1.25, 5) -- ($(1.25, 7.5)$);

\draw[ultra thick, red] 
(2.5, 2.5) -- ($(1.25, 2.5)$);
\draw[dashed, thick, red] 
(2.5, 2.5) -- ($(2.5, 5)$);
\draw[dashed, thick, red] 
(1.25, 5) -- ($(1.25, 2.5)$);

\node at (0, 5) [below right] {$0$};

\node at (1.25, 5) [below right] {$\frac{1}{4}$};

\node at (2.5, 5) [below right] {$\frac{1}{2}$};

\node at (5, 5) [below] {$1$};

\node at (0, 7.5) [left] {$\sqrt{ 2}$};

\node at (0, 2.5) [left] {$-\sqrt{2}$};

\node at (4, 8) {\large $h_2$};
\end{tikzpicture}
\caption{The function $h_2$ that is in the first level of Haar system.}
\label{haar2}
\end{minipage}
\quad \quad
\begin{minipage}{0.45\linewidth}
\begin{tikzpicture}[scale = 0.8]
\draw[thick, black, -latex] 
(0, 5) -- ($(5.5, 5)$);
\draw [thick, black, -latex] 
(0, 2) -- ($(0, 9)$);
\node at (5.5, 5) [below]{$x$};
\node at (0, 9) [left]{$y$};
\draw[ultra thick, red] 
(2.5, 7.5) -- ($(3.75, 7.5)$);
\draw[dashed, thick, red] 
(2.5, 7.5) -- ($(2.5, 5)$);
\draw[dashed, thick, red] 
(3.75, 5) -- ($(3.75, 7.5)$);

\draw[ultra thick, red] 
(5, 2.5) -- ($(3.75, 2.5)$);
\draw[dashed, thick, red] 
(5, 2.5) -- ($(5, 5)$);
\draw[dashed, thick, red] 
(3.75, 5) -- ($(3.75, 2.5)$);

\node at (0, 5) [below right] {$0$};

\node at (5, 5) [below right] {$1$};

\node at (3.75, 5) [below right] {$\frac{3}{4}$};

\node at (2.5, 5) [below right] {$\frac{1}{2}$};

\node at (0, 7.5) [left] {$\sqrt{2}$};

\node at (0, 2.5) [left] {$-\sqrt{2}$};

\node at (4.5, 8) {\large $h_3$};
\end{tikzpicture}
\caption{The function $h_3$ that is in the first level of Haar system.}
\label{haar3}
\end{minipage}
\end{figure}

The basis in $L_2(\R)$ slightly differs from the basis in $L_2[0, 1]$: the function $h_0(x)$ is not included, the shifts in each level are allowed to be along the whole line, and the binary expansions of functions are also added. In other words, we consider the function
$$\psi (t) = \chi_{[0, \frac{1}{2}]} - \chi_{[\frac{1}{2}, 1]},$$
and the basis in $L_2(\R)$ is generated by it  
$$\psi_{j, k} = 2^{j/2} \psi (2^j t - k), \,\, j, k \in \Z.$$

Similarly one defines the multivariate Haar system. We fix a tile $G$ defined by an integer expanding matrix $M$ and by a system of digits $D(M) = \{d_0, d_1, \ldots, d_{m - 1}\}$. Binary contraction is replaced by the multiplication by powers of the matrix $M$. The function $h_0$ is replaced by the indicator of the tile $G$. 

We choose arbitrary $m - 1$ orthonormal vectors $e_1, \ldots, e_{m-1}$ in the space 
$W = \left\{x \in \R^m  \Bigm| \sum_{i = 0}^{m}{x_i} = 0\right\}$, and we denote by $(e_i)_{k}$ the $k$-th coordinate of the vector $e_i$. Firstly, the auxiliary set of Haar wavelet functions is constructed
$$\psi_s(x) = \sqrt{m}\sum \limits_{k = 0}^{m - 1}{(e_s)_{k + 1} \cdot \chi_{M^{-1}(G + d_k)}} \quad \forall {s = 1, \ldots, m - 1}.$$
Then 
$$\left\{m^{j/2}\psi_s(M^j x - k)\right\}_{j \in \Z, \,k \in \Z^d, \,s = 1, \ldots, m-1} $$
is the basis in $L_2(\R^d)$ called \textit{Haar basis}.
Thus, every tile generates a Haar system. 

Orthonormal Haar bases are widely applied in approximation theory, numerical methods, signal processing, etc. (see, for example, \cite{NovProtSkop}, \cite{NovSem}).

\begin{definition}
We call a multivariate Haar system \textit{box-Haar system} if it is generated by a box-tile. 
\end{definition}
As we can see from the construction, the less is the number of digits $m$, the simpler is the corresponding Haar basis. One of the advantages of using box-Haar is the reduction of the number of digits and, as a consequence, of the wavelet functions. This eventually allows us to achieve the required accuracy using less coefficients of expansion.  

For example, there exists the integer matrix $\begin{pmatrix} 0 & 2 \\ 1 & 0 \\\end{pmatrix}$
 with $|\det{M}| = 2$ that generates a box-tile. However, if we restrict ourselves only to tensor products of one-dimensional Haar systems (the simplest systems generated by diagonal matrices), it is impossible to obtain a tile with two digits. In this case one needs at least four digits, and, respectively, three generating wavelet functions (in case of two digits the system is generated by one wavelet-function). Indeed, the matrix $M$ is expanding, so, if it is diagonal, then all diagonal elements must be at least two in absolute value. Therefore, $|\det M| \ge 2\cdot 2 = 4$. Using box-Haars we can guarantee the existence of Haar system with two digits in arbitrary dimension because we can define the corresponding matrix by Theorem \ref{mainbox}. 

Another advantage of using box-Haars is their smoothness. This parameter is responsible for the rate of approximation by the Haar system. We define the H\"older exponent of regularity of function as follows 
$$\alpha_\varphi = \sup\left\{\alpha \geqslant 0 \mid \exists c\colon\; \|\varphi(\cdot) - \varphi(\cdot + h)\|_2 \leqslant c\cdot h^\alpha \right\}.$$
It is known that in $L_2$ the H\"older regularity is equal to Sobolev regularity $$s_\varphi = \sup\left\{s > 0 \mid \int |\hat{\varphi}|^2 (|\xi|^2 + 1)^s d\xi < \infty)\right\},$$ where $\hat{\varphi}(\omega) = \int_{-\infty}^{\infty} \varphi(x) e^{-ix\omega}dx$.

The regularity of Haar system is defined as the regularity of the indicator function of the corresponding tile $G$. It is known that for indicator functions the value of regularity does not exceed $0.5$. So in the case when the tile $G$ is ``simple'': parallelepiped, polyhedra or their union, the regularity of the corresponding Haar system (which is equal to $0.5$ in these cases) is maximal. Thus, the rate of appro\-xi\-ma\-tion by means of box-Haar is the best possible.

\section{The classification of plane polygonal attractors}
\label{plane}
In this section we investigate the case $d = 2$, but move from box-attractors to arbitrary polygonal attractors on the plane. As usual, we assume that a polygon is a plane figure bounded by a closed simple broken line. All polygons are connected but not always convex. Non-convex polygonal attractors are specially difficult to analyse. We show that there are no other polygonal attractors but parallelograms. After that we come back to the case of arbitrary dimension but restrict ourselves to the convex case. 

\begin{theorem}\label{mainnonconvex}
If a plane polygon $G \subset \R^2$ is an attractor, then it is a parallelogram. 
\end{theorem}

Now we introduce the following concept. 

\begin{definition}
We say that a side of the polygon is \textit{extreme}, if the line that contains this side does not intersect the interior of the polygon.  
All sides located on the same line with a given extreme side are called \textit{conjugate}. 
\end{definition}

\begin{proof}
It can be assumed that all angles of the polygon are different from $180^{\circ}$. Let the dilation operator be $M^{-1}$. All sides of the polygon $G$ can be divided into equivalence classes: the sides are equivalent if they are parallel. Consider also such equivalence classes of the polygon $M^{-1}G$. 

The affine transform preserves parallel lines, therefore the numbers of classes of $G$ and of $M^{-1}G$ are equal. Since the polygon $G$ is divided into parallel copies of $M^{-1}G$, each side of $G$ is adjacent to some side of one of the copies $M^{-1}G$. 
Hence, each class of the polygon $G$ has at least one corresponding class of $M^{-1}G$ which is parallel to it. 
Besides, there is the only such class since different classes of $M^{-1}G$ are not parallel. The numbers of classes of $G$ and of $M^{-1}G$ are equal. Let us identify the parallel classes of $G$ and of $M^{-1}G$, then we obtain a permutation $\sigma$ of the classes of a polygon $G$. 
A certain degree of a permutation $\sigma$, say $\sigma^{N_1}$, is equal to the identity. Then $M^{-N_1}$ takes each side of the polygon $G$ to a parallel side of the polygon $M^{-N_1}G$. 

Since $M^{-1}$ is a dilating operator, there exists $N_2$ such that, for every $n > N_2$, the diameter of the set $M^{-n}G$ is at least twice less than the length of an arbitrary side in $G$. In particular, this implies that, for every $n > N_2$, each side of the polygon $G$ is divided by the dilated polygons (copies of $M^{-n}G$) into at least two segments. 

Denote by $N = 2 N_1 \cdot N_2$. Then under the action of $M^{-N}$ on $G$ the image of each side of $G$ is parallel to it, is divided by the contracted polygons into at least two parts. The diameter of the dilated set is at least twice less than the length of an arbitrary side in $G$. The evenness of $N$ will be used later. Let $M_1 = M^{N}$.

Let us prove the following lemma. 
\begin{lemma} \label{lempro}
If a polygon $G$ is an attractor, then at least one of the following sta\-te\-ments holds.  

1) The transform $M_1^{-1} = M^{-N}$ is a homothety with a positive coefficient. 

2) All sides of a polygon $G$ are parallel to two lines. In the coordinate system with axes parallel to these lines, $M_1^{-1}$ is a diagonal matrix with a strictly positive diagonal. 
\end{lemma} 

\begin{proof}
The set of all lines passing through the origin on the plane is a projective line. The matrix $M^{-N / 2}$ defines a projective transform of this line. Besides that, the definition of the number $N$ implies that the operator $M^{-N / 2}$ takes each side of the polygon $G$ to the parallel side of the polygon $M^{-N / 2}G$. If the polygon $G$ has at least three classes of parallel lines, then this projective transform has at least three fixed points. Thus, it is the identity. Consequently, each line on the plane preserves its direction under the transform $M^{-N / 2}$. Thus, $M^{-N / 2}$ is a homothety. Then $M^{-N}$ is a homothety with a positive coefficient, and 1) holds. 

Now consider the second case, when a polygon $G$ has at most two classes of parallel sides. Since the matrix 
$M^{-N / 2}$ takes each side of the polygon $G$ to a parallel side of the polygon $M^{-N / 2}G$, and the sides are parallel to the coordinate axes, the matrix $M^{-N / 2}$ is diagonal. Then $M^{-N}$ is a diagonal matrix with the positive elements on the diagonal. 

Lemma \ref{lempro} is proved. 
\end{proof}

Now continue the proof of Theorem \ref{mainnonconvex}.  
Let us prove that a polygonal attractor has at least one extreme side (notice that it is not necessary for an arbitrary polygon). 

Fix an arbitrary side $s$ of the polygon $G$. It is covered with the non-overlapping sides of small polygons that are adjacent to it. At least one of these segments of nonzero length contains the midpoint of the side $s$. Let it be the side $h$ of the small polygon $H$. The chosen polygon $H$ can not intersect the line that contains $s$ in the points outside $s$ since its diameter is less than $|s| / 2$. Therefore, the whole polygon $H$ lies on the one side of the line that contains $s$, and its side $h$ is extreme. Since affine transform preserves the extreme property of sides, the initial polygon $G$ also has at least one extreme side.

Fix an arbitrary extreme side $l$ of the polygon $G$. Let us introduce the coordinate system with the abscissa along the side $l$ and assume that the whole polygon lies in the upper half-plane. Everywhere below we denote the terms on the left, on the right, up, down in a natural way, according to this coordinate system. Denote the neighbouring sides of the side $l$ of the polygon $G$ by $m_1$ (on the left) and $m_2$ (on the right). 

Only extreme sides of the dilated polygons are adjacent to the extreme side $l$ of the polygon $G$ because each of the contracted polygons is a subset of $G$, that is, a subset of the upper half-plane. Thus, only extreme sides of the dilated polygons which correspond to the extreme sides of the polygon $G$ from the equivalence class of the side $l$ can be adjacent to $l$. Lemma \ref{lempro} implies that it can be only the shifts of the side $M_1^{-1}l$ and its conjugates. Indeed, Lemma \ref{lempro} implies that the transform is a contraction with positive coefficients and does not change the relative position of the sides and their directions. 

Observe that the side $l$ (and similarly an arbitrary extreme side) has no conjugate sides. Assume the converse, then the side $M_1^{-1}l$ has some conjugate sides. Let $T$ be the set of shifts of $M_1^{-1}G$ which tile the polygon $G$. Since the side $l$ is extreme, the polygons from $T$ intersecting $l$ are the shifts of each other along the abscissa. All of them intersect $l$ by several segments that are shifts of the side $M_1^{-1}l$ and of its conjugate sides. Consider one of these polygons from $T$, let it be $F$. The polygon $F$ has at least two segments on the line $l$, and the free space on $l$ between them shoud be filled. Let $H$ be a polygon from $T$ that contains some part of this free space. Then $H$ necessarily intersects $F$, which is a contradiction (see Fig. \ref{nc_2}). 

\begin{figure}[ht]
\begin{minipage}[h]{0.5\linewidth}
\centering
\begin{tikzpicture}[scale=0.9, every node/.style={scale=0.9}]
	\coordinate (X0) at (-2.5cm, -1cm);
	\coordinate (X1) at (-2cm, -1cm);
	\coordinate (X2) at (-1cm, 1cm);
	\coordinate (X3) at (1cm, 0cm);
	\coordinate (X4) at (1.5cm, -1cm);
	\coordinate (X5) at (1cm, -1cm);
	\coordinate (X6) at (0cm, -0.5cm);
	\coordinate (X7) at (-0.5cm, -0.5cm);
	\coordinate (X8) at (-1cm, -1cm);
	\coordinate (X9) at (3.5cm, -1cm);
	\foreach \i in {1,2,...,8}
	{
	  \coordinate (Y\i) at ($(X\i) + (1.4, 0)$);
	}
	\draw[ultra thick, black] (X0)--(X9);
	\draw[ultra thick, blue] (X1)--(X2)--(X3)--(X4)--(X5)--(X6)--(X7)--(X8)--(X1);
	\draw[ultra thick, blue] (X5)--(X4);
	\draw[ultra thick, red] (Y1)--(Y2)--(Y3)--(Y4)--(Y5)--(Y6)--(Y7)--(Y8)--(Y1);
	\draw[ultra thick, red] (Y4)--(Y5);
	\foreach \i in {1,2,...,8}
	{
		\node[draw,circle, blue, inner sep=1pt,fill] at (X\i) {};
		\node[draw,circle, red, inner sep=1pt,fill] at (Y\i) {};
	}
	\node at (barycentric cs:X0=1,X9=1) [below] {$l$};
\end{tikzpicture}
\caption{The proof that the intersection of one polygon with the boundary consists of one segment.}
\label{nc_2}
\end{minipage}
\hfill
\quad
\begin{minipage}[h]{0.5\linewidth}
\centering
\begin{tikzpicture}
	\coordinate (X1) at (-2cm, -1cm);
	\coordinate (X2) at (-2.5cm, 0cm);
	\coordinate (X3) at (-2cm, 1cm);
	\coordinate (X4) at (-1.5cm, 0cm);
	\coordinate (X5) at (-1cm, -1cm);
	\coordinate (X6) at (0cm, -1cm);
	\coordinate (X7) at (-1cm, 1cm);
	\coordinate (X8) at (2cm, -1cm);

	\coordinate (X9) at (-0.5cm, 0cm);
	\coordinate (X10) at (0cm, 1cm);
	\coordinate (X11) at (1cm, -1cm);
	
	\coordinate (X12) at (-3.5cm, 2cm);
	\draw[ultra thick, black] (X1)--(X2)--(X3)--(X5)--(X6)--(X7)--(X4);
	\draw[ultra thick, black] (X1)--(X8);
	\draw[ultra thick, black] (X9)--(X10)--(X11);
	\draw[thick, black] (X1)--(X12);
	
	\draw[thick, red] ($(X1)!.06!(X5)$)--($(X2)!.06!(X4)$);
	\draw[thick, red] ($(X4)!.06!(X9)$)--($(X5)!.06!(X6)$);
	\draw[thick, red] ($(X9) + (0.06, 0)$)--($(X6) + (0.06, 0)$);
	\node at (barycentric cs:X1=1,X4=1) {$H_1$};
	\node at (barycentric cs:X1=1,X12=1)[left] {$m_1$};
	\node at (barycentric cs:X5=1,X9=1) {$H_2$};
	\node at (barycentric cs:X1=1,X2=1)[red, left] {$l_1$};
	\node at (X6)[below] {$l$};
	\foreach \x in {-2,-1,...,2}
	{
		\node[draw,circle,inner sep=1pt,fill] at (\x, -1) {};
	}
	\end{tikzpicture}
\caption{The proof that the neighbours of an extreme side of the polygonal attractor are parallel.}
\label{nc_3}
\end{minipage}

\end{figure}

Besides, if a side of a small polygon is adjacent to $l$, then it is a subset of $l$. Otherwise the small polygon intersects either $m_1$ or $m_2$ (they lie as well as the whole polygon in the upper half-plane).

Hence, the side $l$ is divided into equal segments by the polygons from $T$ (as we know, to at least two segments). Consider the leftmost of these segments (in the direction of the abscissa). Let it be a side of the polygon $H_1$. The polygon $H_1$ has the side $l_1$ which corresponds to the left side $m_1$ of the initial polygon $G$. Lemma \ref{lempro} implies that $l_1$ lies on the line which contains $m_1$. Since the diameter of $H_1$ is less than the sides of a polygon $G$, the side $l_1$ is a subset of $m_1$. 

Let us show that $m_1$ is an extreme side. Suppose the converse, that the line which contains $m_1$ intersects the interior of the polygon $G$. Then the line which contains $l_1$ intersects the interior of the polygon $H_1$. However, since the diameter of $H_1$ is less than the length of the side $m_1$, the side $m_1$  intersect the interior of a polygon $H_1$, which is impossible. Thus, $m_1$ is an extreme side in the polygon $G$. 

Then the left neighbouring side of $m_1$ is also extreme and so on. Following the chain we prove that all sides of a polygon $G$ are extreme. This implies that the polygon is convex.

Now prove that the sides $m_1$ and $m_2$ are parallel. Similarly, the neighbours of an arbitrary extreme side are parallel, i.e. 
every second side of the polygon is parallel. Then because of the convexity the polygon is necessarily a parallelogram. As we found out before, the side $l$ is divided into the shifts of the side $M_1^{-1}l$. The first segment of a partition corresponds to the polygon $H_1$. Suppose that the second segment of a partition on $l$ corresponds to the polygon $H_2$ which is a shifted copy of the polygon $H_1$. Then it has a side equal to the side $l_1$. One of the sides of $H_1$ also lies on the same line with this side of $H_2$ (see Fig. \ref{nc_3}). Otherwise there will be an angle that is impossible to fill with finitely many polygons with a nonzero edge along the abscissa. We obtain that the two neighbouring sides of the segment on $l$ in $H_1$ are parallel. The images of these sides under an affine transform are $m_1$ and $m_2$, therefore $m_1$ is parallel to $m_2$. 

Theorem \ref{mainnonconvex} is proved.   
\end{proof}
Thus, there is no convex polygonal attractors on the plane but parallelograms. 

\section{The classification of the convex attractors}
\label{convexcase}
Thus, we found all polygonal attractors on the plane. Now we come back to the arbitrary dimension but consider only the convex case. A general case including nonconvex polyhedra remains an open problem, see also Section \ref{hyp}. 

\begin{remark}
It is interesting to consider the questions related to the structure of an arbitrary self-affine convex bodies (i.e. having a partition into their affine copies with non-overlapping interiors, $X = \bigcup_{i = 1, \ldots, m}A_i X$, where $A_i$ is an arbitrary affine operator). In the work of C.~Richter \cite{Richter} it was shown that on the plane such sets are always polygons. In the work of A.S.~Voynov \cite{Voyn} these sets were classified in $\R^3$ and G.Valett's conjecture in the dimension $d \ge 3$ was disproved: \textit{Every self-affine body is either a polyhedron or affinely equivalent to a direct product of a self-affine polyhedron and of some convex body of a smaller dimension.} 
\end{remark}

All convex attractors are polyhedra. This statement is well-known. Even stronger statements hold, for instance, Theorem $1$ from \cite{Voyn}: 

\textit{Let $A_1, \ldots, A_m$ be the dilating family of affine operators, i.e. the norm of an arbirary product $A_{t_1}\ldots A_{t_k}$ tends to $0$ as $k \to \infty$ where $t_i \in \{1, \ldots, m\}$. If a convex body $G \subset \R^d$ has a self-affine property, i.e. is a union of its images $A_1G, \ldots, A_mG$ and the intersection of different images have zero measure, then $G$ is a polyhedron. } 

Let $A_i x = M^{-1}(x + s_i)$, then we obtain that every convex attractor is a polyhedron. 

Now prove that the convex attractor that is a polyhedron is always a parallelepiped. Let us remark that it is not true for self-affine convex bodies from \cite{Voyn}.

\begin{theorem}\label{mainconvex}
Every convex attractor $G$ in $\R^d$ is a parallelepiped.
\end{theorem}
\begin{proof}
We can assume that $G$ is a polyhedron. 
Let us prove the theorem by induction on the dimension of the space. For $d = 2$ we already know that the statement holds (Theorem \ref{mainnonconvex}). Suppose it holds for the dimension $d - 1$, then prove it for $d$. 

Fist we argue as in the two-dimensional case. Let $M^{-1}$ be the contraction operator. All hyperfaces of a polyhedron $G$ can be divided into the equivalence classes: the hyperplanes are equivalent if they are parallel. Consider also such classes of equivalence of the polyhedron $M^{-1}G$. The affine transform preserves parallel lines, therefore the  numbers of classes of $G$ and of $M^{-1}G$ are equal. Since the polyhedron $G$ is divided into the copies of $M^{-1}G$, each hyperplane of $G$ is adjacent to some hyperplane of one of the copies $M^{-1}G$.  
Hence, each class of the polyhedron $G$ has at least one corresponding class of $M^{-1}G$ parallel to it. 
Besides, there is the only such class since different classes of $M^{-1}G$ are not parallel. Since the numbers of classes of $G$ and of $M^{-1}G$ are equal, the correspondence is bijective and is given by a permutation $\sigma$ of the classes. 

Some degree of a permutation $\sigma$, let it be $\sigma^{N_1}$, is equal to the identity permutation. Let us remark that because of the convexity of the polyhedron there is at most two faces in each class, since a convex polyhedron has at most one face parallel to a given one.  
Let $N$ be a product $2 \cdot N_1$, then under the action of $M^{-N}$ on $G$ the image of each face of $G$ is associated to it. Since the matrix $M^{-1}$ is dilating, each side of $G$ is divided by the contracted polyhedra into at least two parts. Let $M_1 = M^{N}$. 

Fix an arbitrary face $S$ of the polygon, then it is divided by the shifts of a dilated polyhedron $M_1^{-1}G$ into several similar to it  polyhedra which are the shifted copies of each other. As we reduce the dimension by one, we can use the induction hypothesis and obtain that this face $S$ is a parallelepiped. We do it for each face and obtain that all faces of $G$ are parallelepipeds.

\begin{remark} The fact that all faces of $G$ are parallelepipeds does not imply that $G$ is a parallelepiped. The counterexample is a rhombic dodecahedron.  
\footnote{The author is grateful to the anonymous reviewer for pointing out this example}. 
\end{remark}

Observe that~$G$ additionally possess the following property:  for each $(d-1)$-face (that is, as we know, a parallelepiped) and 
for every pair of its opposite~$(d-2)$-faces, there exist two $(d-1)$-faces parallel to each other that contains them. Then we establish that this property together with the convexity implies that $G$ is a parallelepiped.

Fix an arbitrary $(d - 1)$-face~$H$ and choose arbitrary two of its opposite  $(d-2)$-faces $T_1$ and $T_2$. The polyhedron $G$ has two $(d-1)$-faces that contain $T_1$, $T_2$ and are different from~$H$. 
 Denote them by $K_1$, $K_2$ and prove that they are parallel. 

The face $H$ is adjacent to many $d$-polyhedra that are copies of $M_1^{-1}G$. Their $(d - 1)$-faces-parallelepipeds that lie in $H$ are shifted copies of each other. Consider a $d$-polyhedron $\bar P$ that is adjacent to $H$ by the $(d - 1)$-face $P$ and also has a $(d - 1)$-face that is adjacent to $K_1$, i.e. that is located leftmost 
(Fig. \ref{antirombo}). The parallelepiped $P$ has the $(d - 2)$-face $P_1$ lying in $K_1$ and the only parallel to it $(d - 2)$-face $P_2$. Consider the neighbour of the parallelepiped $H$, let it be $Q$, such that it is adjacent to the $(d - 2)$-face $P_2$ by its $(d - 2)$-face $Q_1$ (not necessarily completely). In the parallelepiped $Q$ the face $Q_1$ has the only parallel $(d - 2)$-face $Q_2$. 
The $d$-polyhedra $\bar P$ and $\bar Q$ corresponding to $P$ and $Q$ have $(d - 1)$-faces $\bar P_1$, $\bar P_2$, $\bar Q_1$, $\bar Q_2$ that contain $P_1$, $P_2$, $Q_1$, $Q_2$ ($\bar P_1$ is along $K_1$). The parallelepipeds $P$ and $Q$ are adjacent to each other by the faces $P_2$, $Q_1$. Let us prove that the polyhedra $\bar P$ and $\bar Q$ are also adjacent to each other by the faces $\bar P_2$, $\bar Q_1$, i.e. that there could not be an angle between the faces $\bar P_2$ and $\bar Q_1$. Suppose the converse. Take some one-dimensional edge $e$ of the parallelepiped $Q$ that is not parallel to its face $Q_1$. Since the face $H$ is filled with the copies of a parallelepiped $Q$, it is possible to draw a segment equal to $e$ through an arbitrary point of the face $H$  such that it lies entirely in one of the copies of $Q$. The polyhedra that are adjacent to $H$ fill the whole polyhedron $G$. Therefore, there is a number $\varepsilon > 0$ such that through an arbitrary point of the polygon $G$ with the distance to the face $H$ less than $\varepsilon$ it is possible to draw a segment equal to $0.5 e$ such that it lies entirely in one of the copies of $M^{-1}G$. 
Let us draw a hyperplane through the faces $P_2$ and $Q_1$ that is parallel to $K_1$. We can draw a segment equal to $0.5e$ through  an arbitrary point of this hyperplane with the distance to the intersection of the faces $P_2$ and $Q_1$ less than $\varepsilon$ so that this segment belongs to the only of tiling polyhedra. On the other hand, if we choose this point close enough to the intersection of the faces, then an arbitrary segment will intersect both polyhedra $\bar P$ and $\bar Q$ since it can not fit entirely in the fixed angle between the faces $\bar P_2$ and $\bar Q_1$, which is a contradiction. 

\begin{figure}[h!]
\centering
\includegraphics[width = 0.6\textwidth]{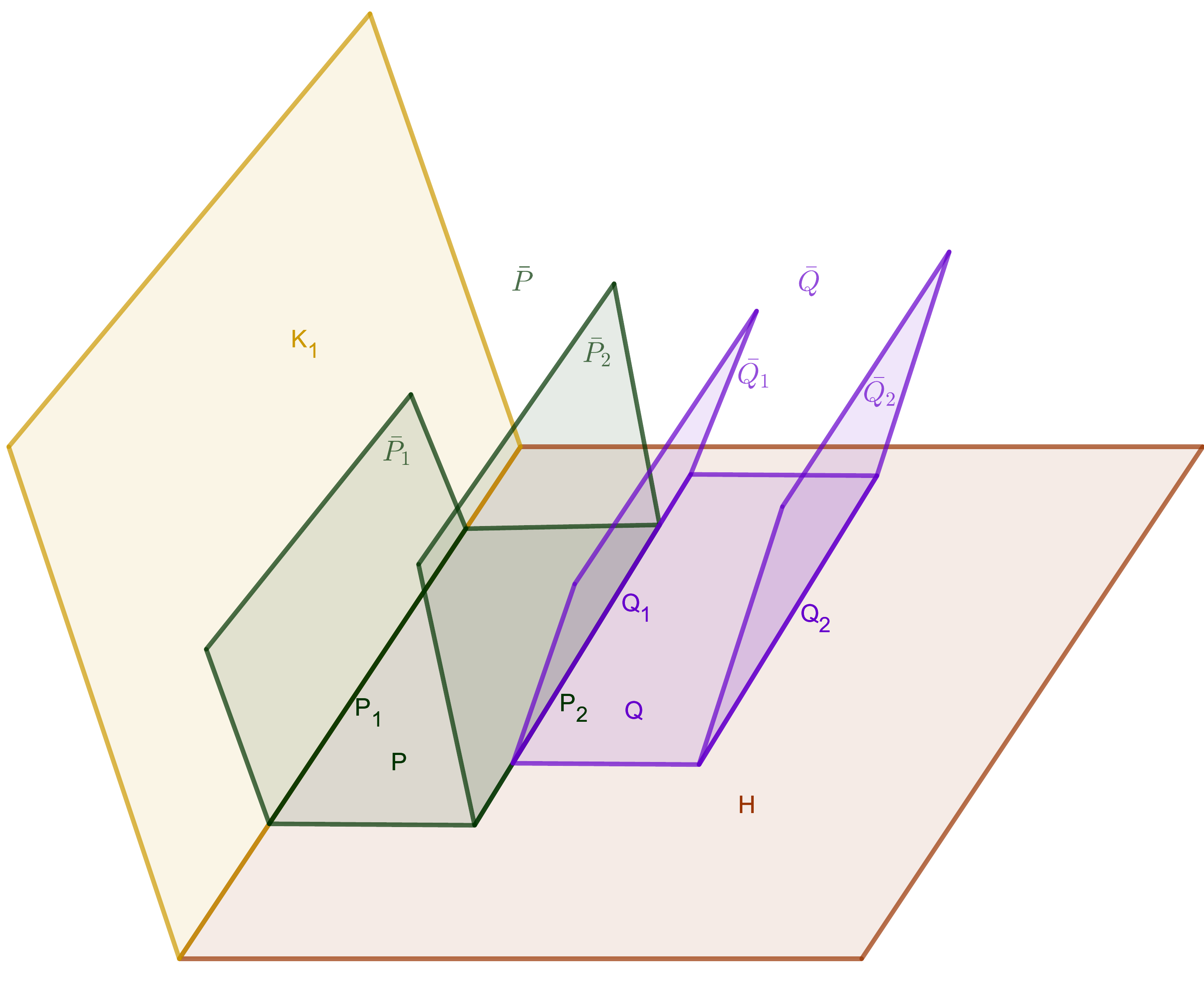}
\caption{The adjacent to each other dilated copies of the polyhedron $G$.}
\label{antirombo}
\end{figure}

So the faces $\bar P_2$ and $\bar Q_1$ are parallel. 
Since $P_1$ and $Q_2$ are such that $\bar P_1$ and $\bar Q_2$ can be obtained from each other with a shift (that takes $\bar P$ to $\bar Q$, $P$ to $Q$, $P_1$ to $Q_1$, and so on), all the $(d - 1)$-faces $\bar P_1$, $\bar P_2$, $\bar Q_1$, $\bar Q_2$ are parallel to each other. Thus, the $d$-polyhedron $\bar P$ has two parallel faces $\bar P_1$ and $\bar P_2$ which are both neighbours of the face-parallelepiped $P$. If we recall that $\bar P$ is the dilated initial polyhedron $G$, we conclude that $K_1$ is parallel to $K_2$.

Thus, we have shown that for each $(d-1)$-face of a polyhedron~$G$ and for each pair of its opposite~$(d-2)$-faces, there exist two $(d-1)$-faces parallel to each other that contain them. Take an arbitrary~$(d-1)$-face $H$. It has exactly 
$2(d-1)$ faces of the dimension~$(d-2)$. Denote them by $T_i,\, i = 1, \ldots , 2(d-1)$. 
Every face~$T_i$ is a subset of the only $(d-1)$-face different from $H$. 
Denote it by $Q_i$. Also denote by $V_i$ the face $Q_i$ of the dimension $(d-2)$ that is opposite to $T_i$. The face $V_i$ is a shifted face $T_i$. Denote the shift vector by $s_i$. These vectors are equal since the neighbouring $(d-1)$-faces $Q_i$, $Q_{i+1}$ have the common edges equal to both $s_i$ and $s_{i+1}$. 
Let us apply the established property for the $(d-1)$-face $Q_i$. Hence, there is two parallel $(d-1)$-faces that contain the pair of the opposite $(d-2)$-faces $T_i$ and  $V_i$.  One of them is $H$. Denote the other by $H_i$.  
Thus, the polyhedron~$G$ has $(d-1)$-faces $H_i, \, i = 1, \ldots , 2(d-1)$, parallel to the face~$H$. 
As $G$ is convex, it has at most one $(d-1)$-face parallel to $H$. Consequently, all $H_i$ coincide. Thus, the $(d-1)$-face $H_1$ that is parallel to $H$ contains all $(d-2)$-faces of $V_i, \, i = 1, \ldots , 2(d-1)$.  Since the $(d-2)$-faces $T_i$ form a parallelepiped, after the shift by the same vector taking them to $V_i$, they will again form a parallelepiped. Therefore, the face $H_1$ is a shift of the parallelepiped $H$. 
 
Every pairs of $(d-2)$-dimensional faces $(T_i, V_i), \, i = 1, \ldots ,2(d-1)$, where  $T_i \subset H, V_i \subset H_1$, is a pair of opposite faces of the hyperface~$Q_i$. Therefore, the following $2d$ parallelepipeds 
($(d-1)$-dimensional) are the hyperfaces of $G$: $H$ (the lower face), $H_1$ (the top face) and $Q_i, \, i = 1, \ldots , Q_{2(d-1)}$ (the lateral faces). Denote by $G'$ the convex hull of $H$ and $H_1$. It is a parallelepiped that is a subset of $G$, all of its $2d$ faces lie inside the faces of $G$. Hence, $G$ and $G'$ coincide, i.e. $G$ is indeed a parallelepiped. 
\end{proof}

\section{Intermediate conclusions}
So, convex attractors other than parallelepipeds do not exist. On the other hand, as we have shown in Section \ref{plane}, on the two-dimensional plane all possible polygonal attractors are parallelepipeds. We do not know if the same is true in higher dimensions, see the discussion in Section \ref{hyp}. Nevertheless, another natural question arises: do other ``simple'' attractors exist? It turns out that there are some but they are disconnected. We start the classification of disconnected attractors in the dimension one, that is, from the line. Even in this case the problem is nontrivial. In higher dimensions we give some examples of simple attractors and discuss the possibility of general classification in Section \ref{hyp}.

It is natural to define simple disconnected attractor as a union of a finite number of polyhedra. On the line it is the finite number of segments.

\section{The classification of attractors on the line}
\label{lineres}
In this section we classify one-dimensional attractors which consist of finitely many segments. They will be referred to as \textit{simple one-dimensional attractors}. The classification of two-dimensional attractors consisting of finitely many polygons remains an open problem, see also the conjecture in Section \ref{hyp}. 

The classification of simple one-dimensional attractors can be reduced to the problem of tiling the integer segment by the shifts of a single set (the classification of integer tiles). This problem was solved by K. Long (see \cite{Long}) in 1967 by applying of a strong combinatorial result of de Bruijn (see \cite{Bruij2}). We give this classification in a somewhat different formulation (clearer from our point of view) with the use of arithmetic progressions. We supply this result with a new autonomous proof. 

In the work \cite{BodRiv} the authors give a classification of integer tiles (apparently independent of Long), obtain the results about periods of tiling, and give an algorithm that calculates the minimal integer tiling interval. There is no proofs in that work. Besides that, the problems related to integer attractors were studied in the works of \cite{Tijdeman} -- \cite{Kol}.

\subsection{Reduction to the combinatorial problem}

\begin{lemma} \label{to_tiling}
A set $W$ consisting of finitely many segments with integer ends is an attractor if and only if some finite set of non-overlapping shifts of the set $W$ filles some segment. 
\end{lemma}
\begin{proof}
Let us suppose that such shifts exist and fill a segment of length $h$, and the set $W$ consists of segments $I_1, I_2, \ldots, I_n$ with integer lengths. Then it is possible to fill a segment of length one with the shifts $W / h$, so it is possible to fill all segments $I_1, I_2, \ldots, I_n$, that is, to fill the set $W$ with its dilated in $h$ times copies. 

Otherwise, having an attractor (not necessarily with integer elements), we can refine its partition to make it small enough, so that each contracted set is a subset of one of the segments of the  initial set. With the shifts of non-overlapping dilated sets, it is possible to fill, for example, the first segment of the set $W$. Expanding the contracted attractors back to the initial size, we obtain the tiling of a segment with the shifts of the initial set. 

Lemma \ref{to_tiling} is proved. 
\end{proof}

\begin{lemma}
If $W$ is an attractor consisting of finitely many segments, then all of its segments have the same length, and the distances between them are multiples of this length. 
\end{lemma}
\begin{proof}
As we know from the proof of Lemma \ref{to_tiling}, several non-overlapping shifts of the set $W$ fill the segment. Without loss of generality we may assume that all shifts are positive, i.e., make the shift to the right. 

By dilating or expanding the picture, we assume that the length of the leftmost segment $I_1$ of $W$ is equal to one. If it is the only segment of the set $W$, the statement is proved. 
Otherwise, to fill the space just after a segment $I_1$,  there is a shift by one in the set of shifts. Let us notice that then the lengths of all shifts are not less than one, otherwise a segment $I_1$ will intersect itself after shifting. 

a) The lengths of all segments of $W$ are at most one, otherwise we get self-intersection of a segment of length more than one with the shift by one (see Fig. \ref{od_1}). 

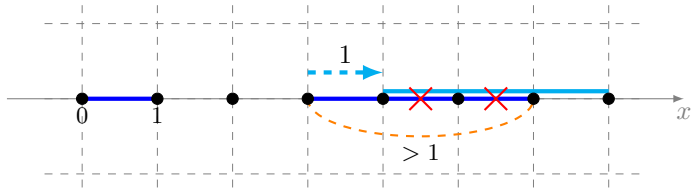
\begin{figure}[ht]
  \centering{
\begin{tikzpicture}
    \coordinate (Origin)   at (0,0);
    \coordinate (XAxisMin) at (-4,0);
    \coordinate (XAxisMax) at (5,0);
    \draw [thin, gray,-latex] (XAxisMin) -- (XAxisMax) node [below] {$x$};
    \clip (-5,-1.3) rectangle (5cm,1.3cm); 
    \pgftransformcm{0.5}{0}{0}{0.5}{\pgfpoint{0cm}{0cm}}
    \draw[style=help lines,dashed] (-7,-7) grid[step=2cm] (9,3);
\draw [ultra thick, blue] ($(-6, 0)$) -- ($(-4, 0)$);
\node at (-6, 0) [below] {\small$0$};
\node at (-4, 0) [below] {\small$1$};
\draw [ultra thick, blue] ($(0, 0)$) -- ($(6, 0)$);
\draw [ultra thick, cyan] ($(2, 0.2)$) -- ($(8, 0.2)$);
\draw [dashed, ultra thick, cyan, -latex] ($(0, 0.7)$) -- ($(2, 0.7)$);
\node at (1, 0.7) [above]{$1$};
\draw [dashed, thick, orange] (0, 0) arc(180:360:3 and 1);
\node at (3, -1.43) {\small{$> 1$}};
\draw [thick, red] ($(2.7, -0.3)$) -- ($(3.3, 0.3)$);
\draw [thick, red] ($(2.7, 0.3)$) -- ($(3.3, -0.3)$);
\draw [thick, red] ($(4.7, -0.3)$) -- ($(5.3, 0.3)$);
\draw [thick, red] ($(4.7, 0.3)$) -- ($(5.3, -0.3)$);
    \foreach \x in {-3,-2,...,4}{
        \node[draw,circle,inner sep=1.5pt,fill] at (2*\x, 0) {};
    }
  \end{tikzpicture}}
\caption{The proof that the lengths of the segments of one-dimensional attractor are at most one.}
\label{od_1}
\end{figure}

b) The lengths of all segments of $W$ are at least one. Indeed, suppose the converse, then we find the leftmost segment $J$ whose length is $a < 1$. After the shift by one we obtain the hole of the size $1 - a$ that is needed to be filled. The only segment lefts that can fit in this hole is the segment $J$. But in this case $J$ is moved by a distance smaller than one, it is a contradiction with the fact that the lengths of all shifts are at least one (see Fig. \ref{od_2}). 

\begin{figure}[ht]
\centering
  \begin{tikzpicture}
    \coordinate (XAxisMin) at (-5,0);
    \coordinate (XAxisMax) at (4,0);
    \draw [thin, gray,-latex] (XAxisMin) -- (XAxisMax) node [below] {$x$};
    \pgftransformcm{0.5}{0}{0}{0.5}{\pgfpoint{0cm}{0cm}}
    \clip (-10,-4) rectangle (8, 4);
\draw[style=help lines,dashed] (-10cm,-4cm) grid[step=3cm] (7,4);
\draw[ultra thick, blue] (-9, 0) -- (-6, 0);
\node at (-9, 0) [below] {$0$};
\node at (-6, 0) [below] {$1$};
\draw[ultra thick, blue] (0, 0) -- (1.2, 0);
\node at (0.6, 0) [above] {$J$};
\draw[ultra thick, cyan] (3, 0) -- (4.2, 0);
\draw[dashed, ultra thick, cyan, -latex] (0, 1.5) -- (3, 1.5);
\node at (1.5, 1.5) [above] {$1$};
\draw [dashed, thick, orange] (0, 0) arc(180:360:0.6 and 0.5);
\draw [dashed, thick, orange] (1.2, 0) arc(180:360:0.9 and 0.5);
\node at (0.6, -0.5)[below] {$a$};
\node at (2.1, -0.5)[below] {\small $1 - a$};
\draw [thick, red] (1.8, -0.3) -- (2.4, 0.3);
\draw [thick, red] (1.8, 0.3) -- (2.4, -0.3);
\foreach \x in {-3,-2,...,2}{
        \node[draw,circle,inner sep=1.5pt,fill] at (3*\x, 0) {};
    }
        \node[draw,circle,inner sep=1.5pt,fill] at (1.2, 0) {};
        \node[draw,circle,inner sep=1.5pt,fill] at (4.2, 0) {};
  \end{tikzpicture}
  \caption{The proof that the lengths of the segments of one-dimensional attractor are at least one.}
\label{od_2}
\end{figure}
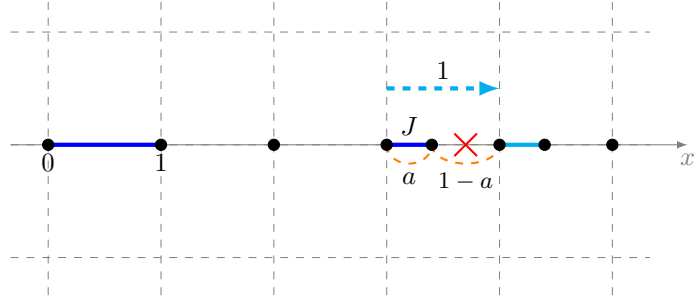

c) All gaps between segments $W$ have to be filled with non-overlapping segments of length one, so apparently all these gaps have integer lengths. 

The lemma is proved. 
\end{proof}

\begin{remark} \label{add_remark} 
The statement of Lemma \ref{to_tiling} holds for sets $W$ with not only integer ends of the segments. We have already proved it in one direction, let us show it in another one. 
Now we know that the lengths of all segments are equal, let their length be $l$. Suppose we tile the segment of length $l \cdot m$; then dilating the set in $m$ times we can fill with non-overlapping shifts an arbitrary segment of length $l$ and thus the whole set $W$. 
\end{remark}

Thus, with the use of similarity, we reduced the classification of one-dimensional attractors consisting of a finite number of segments to the classification of such finite sets of integer segments of length one on the line that some of their positive non-overlapping shifts fill the segment on the line, i.e. to the problem solved in \cite{Long}.  Below we give another proof of this result.  

Firstly consider the case when $W$ is the only segment of  length one, it is a simple one-dimensional attractor; in that follows we suppose that there are at least two segments. We also suppose that $W$ starts at the point $0$. For convenience, we will include the shift by $0$ in the set of shifts of $W$.

\begin{definition} We say that $W$ is an \textit{admissible set} if it is a finite set including $0$ and consisting of at least two integer segments of length one on the line, and there exists a set of shifts $L \ni 0$ such that the shifts of $W$ by the vectors of $L$ are non-overlapping and fill a segment on the line. 
\end{definition} 

Observe that $\{1, \ldots, l - 1\} \subset L$ since it is the only possible variant of the full tiling of a gap between first two segments (Fig. \ref{od_3}). All distances between all segments have to be at least $l$ to avoid overlapping of the segments by these shifts. 

\begin{figure}[ht]
  \centering
  \begin{tikzpicture}
    \coordinate (XAxisMin) at (-3.5, 0);
    \coordinate (XAxisMax) at (4, 0);
    \draw [thin, gray,-latex] (XAxisMin) -- (XAxisMax) node [below] {$x$};
    \pgftransformcm{0.5}{0}{0}{0.5}{\pgfpoint{0cm}{0cm}}
    \clip (-7,-3) rectangle (8, 3);
\draw[style=help lines,dashed] (-7, -3) grid[step=2cm] (7,3);
\draw[ultra thick, blue] (-6, 0) -- (-4, 0);
\node at (-6, 0) [below] {$0$};
\node at (-4, 0) [below] {$1$};
\draw[dashed, ultra thick, orange] (-4, 0) -- (-2, 0);
\draw[semithick, orange, -latex] (-6, 0.5) -- (-4, 0.5);
\draw[semithick, cyan, -latex] (-6, 1) -- (-2, 1);
\draw[semithick, red, -latex] (-6, 1.5) -- (0, 1.5);
\draw[semithick, green, -latex] (-6, 2) -- (2, 2);
\draw[dashed, ultra thick, cyan] (-2, 0) -- (0, 0);
\draw[dashed, ultra thick, red] (0, 0) -- (2, 0);
\draw[dashed, ultra thick, green] (2, 0) -- (4, 0);
\draw[ultra thick, blue] (4, 0) -- (6, 0);
\node at (4, 0) [below] {$l$};
\node at (-4, 0.5) [right] {\tiny $1$};
\node at (-2, 1) [right] {\tiny $2$};
\node at (2, 2) [right] {\tiny $l - 1$};
\foreach \x in {-3,-2,...,3}{
        \node[draw,circle,inner sep=1.5pt,fill] at (2*\x, 0) {};
    }
    \end{tikzpicture}
  \caption{All shifts of the segment $[0, 1]$ by ${1,\ldots ,l - 1}$ fill the first gap between the segments of an attractor.} 
\label{od_3}
\end{figure}
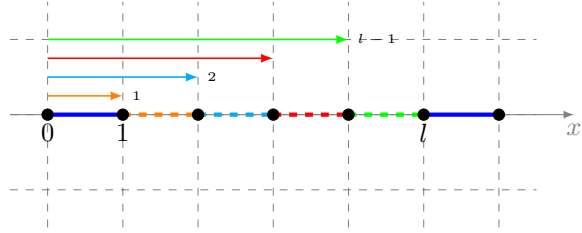

\begin{example}\label{onedimexam}
The example of the set $W$ and its shifts which satisfy all requirements is given in Fig. \ref{od_4}. In this case $L = \{0, 1, 2, 9, 10, 11\}$.
\end{example}

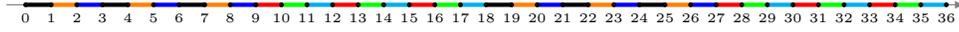
\begin{figure}[ht]
  \centering
  \begin{tikzpicture} [scale = 0.85]
    \coordinate (XAxisMin) at (-9, 0);
    \coordinate (XAxisMax) at (6, 0);
    \draw [thin, gray,-latex] (XAxisMin) -- (XAxisMax);
    \pgftransformcm{0.5}{0}{0}{0.5}{\pgfpoint{0cm}{0cm}}

\draw[ultra thick, black] (-18 * 0.8 - 3, 0) -- (-17 * 0.8 - 3, 0);
\draw[ultra thick, black] (-15 * 0.8 - 3, 0) -- (-14 * 0.8 - 3, 0);
\draw[ultra thick, black] (-12 * 0.8 - 3, 0) -- (-11 * 0.8 - 3, 0);
\draw[ultra thick, black] (0 * 0.8 - 3, 0) -- (1 * 0.8 - 3, 0);
\draw[ultra thick, black] (3 * 0.8 - 3, 0) -- (4 * 0.8 - 3, 0);
\draw[ultra thick, black] (6 * 0.8 - 3, 0) -- (7 * 0.8 - 3, 0);

\draw[ultra thick, orange] (-17 * 0.8 - 3, 0) -- (-16 * 0.8 - 3, 0);
\draw[ultra thick, orange] (-14 * 0.8 - 3, 0) -- (-13 * 0.8 - 3, 0);
\draw[ultra thick, orange] (-11 * 0.8 - 3, 0) -- (-10 * 0.8 - 3, 0);
\draw[ultra thick, orange] (1 * 0.8 - 3, 0) -- (2 * 0.8 - 3, 0);
\draw[ultra thick, orange] (4 * 0.8 - 3, 0) -- (5 * 0.8 - 3, 0);
\draw[ultra thick, orange] (7 * 0.8 - 3, 0) -- (8 * 0.8 - 3, 0);

\draw[ultra thick, blue] (-16 * 0.8 - 3, 0) -- (-15 * 0.8 - 3, 0);
\draw[ultra thick, blue] (-13 * 0.8 - 3, 0) -- (-12 * 0.8 - 3, 0);
\draw[ultra thick, blue] (-10 * 0.8 - 3, 0) -- (-9 * 0.8 - 3, 0);
\draw[ultra thick, blue] (2 * 0.8 - 3, 0) -- (3 * 0.8 - 3, 0);
\draw[ultra thick, blue] (5 * 0.8 - 3, 0) -- (6 * 0.8 - 3, 0);
\draw[ultra thick, blue] (8 * 0.8 - 3, 0) -- (9 * 0.8 - 3, 0);

\draw[ultra thick, red] (-9 * 0.8 - 3, 0) -- (-8 * 0.8 - 3, 0);
\draw[ultra thick, red] (-6 * 0.8 - 3, 0) -- (-5 * 0.8 - 3, 0);
\draw[ultra thick, red] (-3 * 0.8 - 3, 0) -- (-2 * 0.8 - 3, 0);
\draw[ultra thick, red] (9 * 0.8 - 3, 0) -- (10 * 0.8 - 3, 0);
\draw[ultra thick, red] (12 * 0.8 - 3, 0) -- (13 * 0.8 - 3, 0);
\draw[ultra thick, red] (15 * 0.8 - 3, 0) -- (16 * 0.8 - 3, 0);

\draw[ultra thick, green] (-8 * 0.8 - 3, 0) -- (-7 * 0.8 - 3, 0);
\draw[ultra thick, green] (-5 * 0.8 - 3, 0) -- (-4 * 0.8 - 3, 0);
\draw[ultra thick, green] (-2 * 0.8 - 3, 0) -- (-1 * 0.8 - 3, 0);
\draw[ultra thick, green] (10 * 0.8 - 3, 0) -- (11 * 0.8 - 3, 0);
\draw[ultra thick, green] (13 * 0.8 - 3, 0) -- (14 * 0.8 - 3, 0);
\draw[ultra thick, green] (16 * 0.8 - 3, 0) -- (17 * 0.8 - 3, 0);

\draw[ultra thick, cyan] (-7 * 0.8 - 3, 0) -- (-6 * 0.8 - 3, 0);
\draw[ultra thick, cyan] (-4 * 0.8 - 3, 0) -- (-3 * 0.8 - 3, 0);
\draw[ultra thick, cyan] (-1 * 0.8 - 3, 0) -- (0 * 0.8 - 3, 0);
\draw[ultra thick, cyan] (11 * 0.8 - 3, 0) -- (12 * 0.8 - 3, 0);
\draw[ultra thick, cyan] (14 * 0.8 - 3, 0) -- (15 * 0.8 - 3, 0);
\draw[ultra thick, cyan] (17 * 0.8 - 3, 0) -- (18 * 0.8 - 3, 0);

\foreach \x in {-18,-17,...,18}{
        \node[draw,circle,inner sep=0.5pt,fill] at (\x * 0.8 - 3, 0) {};
}
\foreach \x in {0,1,...,36}{
	\node at (\x * 0.8 - 18 * 0.8 - 3, 0) [below] {\tiny{\x}};
    }
    \end{tikzpicture}
  \caption{The example of one-dimensional attractor and its shifts.}
\label{od_4}
\end{figure}

\subsection{Auxiliary result}
\begin{definition} We say that a set of integer numbers $Q$ is an $l$\textit{-set}, if it has finitely many segments (consecutive integer numbers), the beginning of each segment (the first number in the sequence of consecutive numbers) is divisible by $l$, and the lengths of segments (the amount of consecutive numbers in blocks) are equal to $l$. 
\end{definition}

In the example \ref{onedimexam}, $L$ is a 3-set. 

\begin{theorem} \label{doubleind}
Let $W \subset \R$ be an admissible attractor.   
Then the following hold for every $k \in \Z \ge 0$: 

1) the beginnings of all segments in $W$ that do not exceed $k$ are divisible by $l$; 

2) $L \cap \{0, 1, \ldots, k, k + 1\}$ is an $l$-set intersected with $\{0, 1, \ldots, k, k + 1\}$.
\end{theorem}
\begin{proof}
The proof is by induction on $k$. 

\textsl{Base}: $k = 0$. Then in 1) there is only one beginning $0$, and it is divisible by $l$; in 2) we obtain $L \cap \{0, 1, \ldots, k, k + 1\} = \{0, 1\}$, we can choose $\{0, 1, \ldots, l - 1\}$ as an $l$-set . 

\textsl{Step}: suppose the induction hypothesis holds for $k$, let us establish it for $k + 1$. 

To check 1) it is enough to prove that if $k + 1$ is the beginning of the segment $I$, then $k + 1$ is divisible by $l$. Suppose that $k + 1$ is not divisible by $l$, $k + 1 = l\cdot h + r$, $r \neq 0$, $r < l$ (Fig. \ref{od_5}).

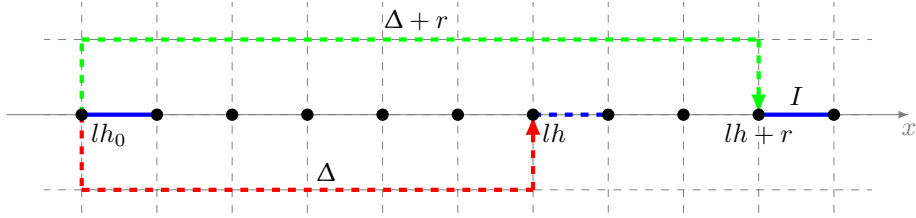
\begin{figure}[ht]
  \centering
  \begin{tikzpicture}
    \coordinate (Origin)   at (0,0);
    \coordinate (XAxisMin) at (-6,0);
    \coordinate (XAxisMax) at (6,0);
    \draw [thin, gray,-latex] (XAxisMin) -- (XAxisMax) node [below] {$x$};
    \clip (-5.5,-1.3) rectangle (5.5cm,1.7cm); 
    \pgftransformcm{0.5}{0}{0}{0.5}{\pgfpoint{0cm}{0cm}}
    \draw[style=help lines,dashed] (-14,-14) grid[step=2cm] (12,3);
   \draw [ultra thick, blue] ($(-10, 0)$) -- ($(-8, 0)$);
\node at (-10, 0) [below right]{$lh_0$};
\draw [dashed, ultra thick, blue] ($(2, 0)$) -- ($(4, 0)$);
\draw [dashed, ultra thick, red] ($(-10, 0)$) -- ($(-10, -2)$);
\draw [dashed, ultra thick, red] ($(-10, -2)$) -- ($(2, -2)$);
\draw [dashed, ultra thick, red, -latex] ($(2, -2)$) -- ($(2, 0)$);
\node at (-4, -2) [above right] {$\Updelta$};
\node at (2, 0) [below right]{$lh$};
\draw [ultra thick, blue] ($(8, 0)$) -- ($(10, 0)$);
\draw [dashed, ultra thick, green] ($(-10, 0)$) -- ($(-10, 2)$);
\draw [dashed, ultra thick, green] ($(-10, 2)$) -- ($(8, 2)$);
\draw [dashed, ultra thick, green, -latex] ($(8, 2)$) -- ($(8, 0)$);
\node at (0, 2) [above left] {$\Updelta + r$};
\node at (8, 0) [below]{$lh + r$};
\node at (9, 0) [above] {$I$};
    \foreach \x in {-7,-6,...,5}{
        \node[draw,circle,inner sep=1.5pt,fill] at (2*\x, 0) {};
    }
  \end{tikzpicture}
  \caption{The proof of the induction step of the statement 1) of the theorem.}
\label{od_5}
\end{figure}

The set $W$ does not have a segment with the beginning in the point $l \cdot h$. Otherwise, if the segment $[lh, lh+1]$ was a subset of $W$, as its distance to $I$ is less than $l$, it would translate with the shift by $r\in L$ to $I \subset W$, leading to self-intersection. So, there is a shift by $\Updelta$ in $L$ that translates one of the previous segments to the point $l \cdot h$. 
By the induction hypothesis 1), this previous segment starts in the point $l \cdot h_0$, $\Updelta = lh - lh_0$, i.e. $\Updelta$ is divisible by $l$; notice also that $\Updelta \le k$.

\begin{lemma} Under the assumptions of Theorem \ref{doubleind}, we have 
$$\{\Updelta + 1, \ldots, \Updelta + l - 1\} \cap \{0, \ldots, k + 1\} \subset L \cap \{0, \ldots, k + 1\}.$$
\end{lemma}
\begin{proof}
By the induction hypothesis 1) of Theorem \ref{doubleind}, $L \cap \{0, \ldots, k + 1\}$ is the beginning of an $l$-set, therefore $\Updelta \in L\cap \{0, \ldots, k + 1\}$ lies in one of the segments of this $l$-set, suppose that in $L_1$. Due to the fact that $L_1$ is a segment of an $l$-set, the only number in it that is divisible by $l$ is its beginning. As $\Updelta$ is a multiple of $l$, exactly $\Updelta$ is the beginning of this segment, then 

\begin{gather*}
\{\Updelta, \Updelta + 1, \ldots, \Updelta + l - 1\} \cap \{0, \ldots, k + 1\} \\
= L_1 \cap \{0, \ldots, k + 1\} \subset L \cap \{0, \ldots, k + 1\}.
\end{gather*} 
The lemma is proved. 
\end{proof} 

\medskip

Observe that $\Updelta + r = lh - lh_0 + r \le lh + r = k + 1$, therefore $\Updelta + r \in \{\Updelta + 1, \ldots, \Updelta + l - 1\} \cap \{0, \ldots, k + 1\}$, then $\Updelta + r \in L$. However, with this shift by $\Updelta + r$ the segment with the beginning $lh_0$ moves to the segment $I$; that is a contradiction with the fact that $I \subset W$.

Let us check 2). Suppose the converse, that for $k + 2$ the statement 2) does not hold (as we know, it holds for $k + 1$). 

\textbf{Case 1:} $k + 2 \notin L$. The condition 2) is no longer true only if $k + 1 \in L$ and the segment of $L$ that contains $k + 1$ has the length less than $l$. In this case, $k + 2$ is not divisible by $l$. 

The beginning of the second segment of $W$ translates to the point $k + 1 + l$ with the shift by $k + 1 \in L$, thus, the whole space left to the $k + 2 + l$ will be eventually filled. 

\textit{Subcase A:} the segment that starts at $k + 2$ (let it be $J$) lies in $W$. 

Consider the nearest to the point $k + 2$ points $A$ (lefts) and $B$ (rights) such that they are divisible by $l$; they do not coincide with $k + 2$ since $k + 2$ is not a multiple of $l$ (Fig. \ref{od_6}).

\begin{figure}[ht]
  \centering
  \begin{tikzpicture}
    \coordinate (Origin)   at (0,0);
    \coordinate (XAxisMin) at (-4,0);
    \coordinate (XAxisMax) at (4,0);
    \draw [thin, gray,-latex] (XAxisMin) -- (XAxisMax) node [below] {$x$};
    \clip (-3.5,-1.3) rectangle (3.5cm,1.3cm); 
    \pgftransformcm{0.5}{0}{0}{0.5}{\pgfpoint{0cm}{0cm}}
    \draw[style=help lines,dashed] (-10,-10) grid[step=2cm] (12,3);
\draw [ultra thick, blue] ($(-2, 0)$) -- ($(0, 0)$);
\draw [dashed, thick, red] (6, 0) arc(0:180:6 and 1.3);
\node at (0.2, 1.3) [above]{$l$};
\draw [dashed, thick, red] (-2, 0) arc(180:360:4 and 1);
\node at (2, -1.42) {\small{$< l$}};
\node at (-1, 0) [above]{$J$};
\node at (-2, 0) [below]{\small{$k + 2$}};
\node at (-6, 0) [below right]{$A$ };
\node at (6, 0) [below right] {$B$};

\draw [ultra thick, blue] ($(8, 0)$) -- ($(10, 0)$);
    \foreach \x in {-4,-3,...,3}{
        \node[draw,circle,inner sep=1.5pt,fill] at (2*\x, 0) {};
    }
  \end{tikzpicture}
  \caption{The case 1,A in the proof of the induction step.}
\label{od_6}
\end{figure}
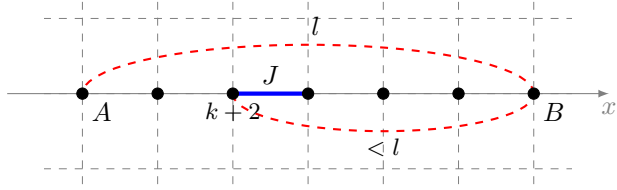

The coordinate $B + 1$ does not exceed $k + 2 + l$, since $[B, B + 1]$ definitely lies in the segment that we will eventually fill. 

As the distance from $k + 2$ to $B$ is less than $l$, $J \subset W$, and $\{1, \ldots, l - 1\} \subset L$, there is a shift in $L$ translating $J$ to $[B, B + 1]$. 

On the other hand, notice that $A \in L$, since $A$ is the beginning of the segment in $L$ that contains $k + 1$. The set $W$ begins with two segments $[0, 1]$ and $[l, l + 1]$. The second one translates with the shift by $A$ to the segment $[B, B + 1]$, i.e. we obtain this segment with the two different shifts from $L$; which is a contradiction. 

\textit{Subcase B:} $[k + 2, k + 3]$ is filled with the shift of some previous segment $I \subset W$. 

The segment $[k + 1, k + 2]$  is not a subset of $W$ since  the segment $[0, 1]$ translates there with the shift by $k + 1$. Thus, the beginning of $I$ is not larger than $k$; then by the induction hypothesis 1) $I = [lh, lh + 1]$. 

Then $k + 2 \notin L$, hence $h \ne 0$, thus, the size of the shift $\Updelta$ equals to $k + 2 - lh \le k + 2 - l \le k + 1$. 
Thus, $\Updelta \in L \cap \{0, 1 \ldots k + 1\}$, this set is the beginning of an $l$-set, therefore $\Updelta$ lies in $L$ in some segment, moreover, not in its beginning as $\Updelta = (k + 2) - lh$ is not divisible by $l$. Then $\Updelta - 1$ lies in the same segment $L$, i.e. $\Updelta - 1 \in L$. Then with the shift by $\Updelta - 1$ $I = [lh, lh + 1]$ translates to $[k + 1, k + 2]$. However, the segments $[0, 1]$ also moves there with the shift by $k + 1 \in L$; which is a contradiction (Fig. \ref{od_7}). 

\begin{figure}[ht]
  \centering
  \begin{tikzpicture}
    \coordinate (XAxisMin) at (-5.5, 0);
    \coordinate (XAxisMax) at (4.5, 0);
    \draw [thin, gray,-latex] (XAxisMin) -- (XAxisMax) node [below] {$x$};
    \pgftransformcm{0.5}{0}{0}{0.5}{\pgfpoint{0cm}{0cm}}
    \clip (-11,-3) rectangle (9, 4);
\draw[style=help lines,dashed] (-11, -3) grid[step=2cm] (8.5,3);
\draw[ultra thick, blue] (-10, 0) -- (-8, 0);
\node at (-10, 0) [below] {$0$};
\node at (-8, 0) [below] {$1$};
\draw[ultra thick, blue] (-4, 0) -- (-2, 0);
\node at (-4, 0) [below right] {\tiny$lh$};
\node at (-2, 0) [below] {\tiny$lh + 1$};
\node at (-3, 0) [above] {$I$};
\draw[ultra thick, cyan] (4, 0) -- (6, 0);
\draw[ultra thick, orange] (6, 0) -- (8, 0);
\draw[dashed, ultra thick, cyan] (-10, 0) -- (-9.7, 1.2) -- (3.3, 1.2);
\draw[dashed, ultra thick, cyan, -latex] (3.3, 1.2) -- (4, 0);
\draw[dashed, ultra thick, orange] (-4, 0) -- (-4, -2) -- (6, -2);
\draw[dashed, ultra thick, orange, -latex] (6, -2) -- (6, 0);
\draw [dashed, ultra thick, red, -latex] (-4, 0) arc(180:0:4 and 2.8);
\node at (0, 3.2) {\small$\Updelta - 1$};
\node at (1, -2) [above] {$\Updelta$};
\node at (4.3, 0) [below] {\tiny$k + 1$};
\node at (6, 0) [below right] {\tiny$k + 2$};
\foreach \x in {-5,-4,...,4}{
        \node[draw,circle,inner sep=1.5pt,fill] at (2*\x, 0) {};
    }
    \end{tikzpicture}
  \caption{The case 1,B in the proof of the induction hypothesis.}
\label{od_7}
\end{figure}
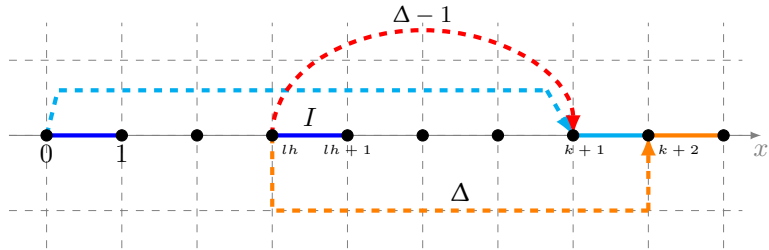

\textbf{Case 2:} $k + 2 \in L$. 

\textit{Subcase A:} $k + 2$ is divisible by $l$. 

Then $k + 2 = lh$. We have a contradiction with 2) for $k + 2$ only in the case if $L$ already contains $l(h - 1), \ldots, l(h - 1) + l - 1$, and the addition of $k + 2$ makes the segment from $L$ longer than $l$. However, the first segments of the set $W$ $[0, 1]$ and $[l, l + 1]$ translate with the shift by $l(h - 1)$ to $[l(h - 1), l(h - 1) + 1]$ and $[lh, lh + 1]$, so we obtain the segment $[lh, lh + 1]$ in two ways -- from the segment $[0, 1]$ and from the segment $[l, l + 1]$; which is a contradiction. 

\textit{Subcase B:} $k + 2$ is not divisible by $l$. 

Statement 2) is true for $k + 1$, therefore we can have a contradiction for $k + 2$ only in the case if $k + 1 \notin L$. 

The segment $[k + 1, k + 2] \not\subset W$, otherwise it translates to $[k + 2, k + 3]$ with the shift by one, and we obtain $[k + 2, k + 3]$ in two ways, from $[0, 1]$ and $[k + 1, k + 2]$ (Fig. \ref{od_8}). 
\begin{figure}[ht]
  \centering
  \begin{tikzpicture}
    \coordinate (XAxisMin) at (-3.5, 0);
    \coordinate (XAxisMax) at (4, 0);
    \draw [thin, gray,-latex] (XAxisMin) -- (XAxisMax) node [below] {$x$};
    \pgftransformcm{0.5}{0}{0}{0.5}{\pgfpoint{0cm}{0cm}}
    \clip (-7,-3) rectangle (8, 3);
\draw[style=help lines,dashed] (-7, -3) grid[step=2cm] (7,3);
\draw[ultra thick, blue] (-6, 0) -- (-4, 0);
\node at (-5.7, 0) [below] {$0$};
\node at (-3.7, 0) [below] {$1$};
\draw[ultra thick, blue] (2, 0) -- (4, 0);
\draw[dashed, ultra thick, red] (4, 0) -- (6, 0);
\draw[dashed, ultra thick, orange] (-6, 0) -- (-6, -2) -- (4, -2);
\draw[dashed, ultra thick, orange, -latex] (4, -2) -- (4, 0);
\draw[dashed, ultra thick, cyan] (2, 0) -- (2, 1) -- (4, 1);
\draw[dashed, ultra thick, cyan, -latex] (4, 1) -- (4, 0);
\node at (2, 0) [below] {\small$k + 1$};
\node at (4.3, 0) [below] {\small$k + 2$};
\node at (-1, -2) [above] {\small $k + 2$};
\node at (3, 1) [above] {$1$};
\foreach \x in {-3,-2,...,3}{
        \node[draw,circle,inner sep=1.5pt,fill] at (2*\x, 0) {};
    }
    \end{tikzpicture}
  \caption{The case 2, B (part 1) in the proof of the induction hypothesis.}
\label{od_8}
\end{figure}
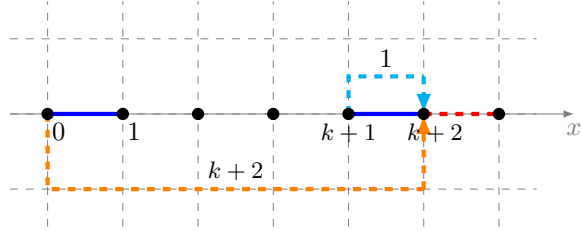

So, $[k + 1, k + 2]$ can be obtained with some shift from $[lh_0, lh_0 + 1]$ (as by the induction hypothesis 1) the beginnings of segments which are less than $k$ are divisible by $l$). Let $\Updelta = k + 1 - lh_0$ be the length of this shift. The number $h_0 \ne 0$ since $k + 1 \notin L$, thus, $\Updelta \le k + 1 - l$; than the whole segment in $L$ which contains $\Updelta$ also lies in $L \cap \{0, 1, \ldots, k + 1\}$. As we consider the subcase $B$,  $\Updelta + 1$ is not divisible by $l$ and that is why it lies in the same segment in $L \cap \{0, 1, \ldots, k + 1\}$. The segment $[lh_0, lh_0 + 1]$ translates to the segment $[k + 2, k + 3]$ with the shift by $\Updelta + 1$; it is a contradiction with the fact that we also obtain it from $[0, 1]$ (Fig. \ref{od_9}).
Theorem is proved. 

\begin{figure}[ht]
  \centering
  \begin{tikzpicture}
    \coordinate (XAxisMin) at (-5.5, 0);
    \coordinate (XAxisMax) at (4.5, 0);
    \draw [thin, gray,-latex] (XAxisMin) -- (XAxisMax) node [below] {$x$};
    \pgftransformcm{0.5}{0}{0}{0.5}{\pgfpoint{0cm}{0cm}}
    \clip (-11,-3) rectangle (9, 4);
\draw[style=help lines,dashed] (-11, -3) grid[step=2cm] (8.5,3);
\draw[ultra thick, blue] (-10, 0) -- (-8, 0);
\node at (-9.7, 0) [below] {$0$};
\node at (-7.7, 0) [below] {$1$};
\draw[ultra thick, blue] (-4, 0) -- (-2, 0);
\node at (-4, 0) [below right] {\tiny$lh_0$};
\node at (-2, 0) [below right] {\tiny$lh_0 + 1$};
\draw[ultra thick, cyan] (4, 0) -- (6, 0);
\draw[ultra thick, orange] (6, 0) -- (8, 0);
\draw[dashed, ultra thick, cyan] (-4, 0) -- (-4, 1) -- (4, 1);
\draw[dashed, ultra thick, cyan, -latex] (4, 1) -- (4, 0);
\draw[dashed, ultra thick, orange] (-10, 0) -- (-10, -2) -- (6, -2);
\draw[dashed, ultra thick, orange, -latex] (6, -2) -- (6, 0);
\draw [dashed, ultra thick, red, -latex] (-4, 0) arc(180:0:5 and 2.8);
\node at (1.6, 2.3) {\small$\Updelta + 1$};
\node at (0, 1.4) {\small$\Updelta$};
\node at (4.3, 0) [below] {\tiny$k + 1$};
\node at (6, 0) [below right] {\tiny$k + 2$};
\foreach \x in {-5,-4,...,4}{
        \node[draw,circle,inner sep=1.5pt,fill] at (2*\x, 0) {};
    }
    \end{tikzpicture}
  \caption{The case 2, B (part 2) in the proof of the induction hypothesis.}
\label{od_9}
\end{figure}
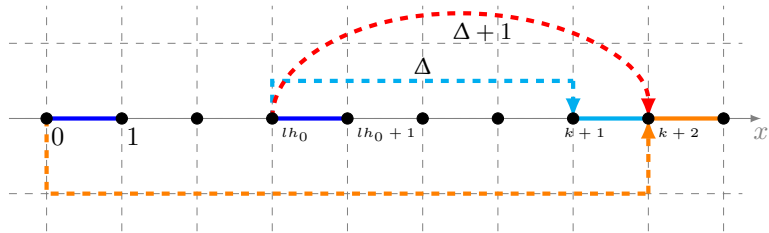
\end{proof}

\subsection{The fundamental theorem on one-dimensional attractors} Now, using Theorem \ref{doubleind}, we figure out the structure of attractors on the line. We consider attractors with integer ends of the segments, with the beginning in zero, we do not restrict the lengths of the segments to one. 

We switch from attractors consisting of segments and tiling with shifts a segment of the line to attractors consisting of integer numbers tiling with shifts a set of consecutive integer numbers (we will call such set a \textit{segment of integer numbers}). Specifically, let $W$ be an attractor consisting of segments. Let us define the integer attractor $Y$ by the rule 
$$z \in Y \Leftrightarrow [z, z + 1] \subset W.$$
The set $Y$ also has the beginning in zero and is not empty. If the attractor $W$ tiles with some set of shifts the segment $[0, y_0 + 1]$, then with the same set of shifts the attractor $Y$ tiles the segment of integer numbers $\{0, 1, \ldots, y_0\}$ (and thus it is an attractor). Conversely, by the attractor of integer numbers $Y$ with the beginning in zero which filles with the set of its shifts $L$ a segment of integer numbers $\{0, 1, \ldots, y_0\}$, the set  
$$W =\bigcup \limits_{y \in Y} {[y, y + 1]}$$ 
consisting of segments with integer ends is uniquely determined. Shifting the set $W$ by the elements of the set $L$, we obtain a tiling of the segment $[0, y_0 + 1]$ with one layer. So, the set $W$ is an attractor consisting of segments. Thus, these problems of classification of attractors are equivalent.   

\begin{definition}
Suppose $A = \{a_1, \ldots, a_n\}$, $B = \{b_1, \ldots, b_k\}$ and for every different pairs $(i_1, j_1)$ and $(i_2, j_2)$ we have $a_{i_1} + b_{j_1} \ne a_{i_2} + b_{j_2}$. Then we define \textit{the sum of the sets} $A$ and $B$ as
$$A \oplus B = \{a_i + b_j \mid i = 1 \ldots n, j = 1 \ldots k\}.$$
\end{definition}

\begin{lemma}
If $A \oplus B = A \oplus C$, then $B = C$.
\end{lemma}
\begin{proof}
Let 

$$A = \{a_1, \ldots, a_n\}, \quad a_1 < \ldots < a_n,$$
$$B = \{b_1, \ldots, b_k\}, \quad b_1 < \ldots < b_k,$$
$$C = \{c_1, \ldots, c_m\}, \quad c_1 < \ldots < c_m.$$ 
Assume the converse, let 
$$b_1 = c_1, \ldots, b_p = c_p, b_{p + 1} \ne c_{p + 1}, \quad p < k, p < m.$$   
Since the sum of the sets is direct, the minimal element in $A \oplus B \setminus \{a_1 + b_1, \ldots, a_n + b_p\}$ is equal to $a_1 + b_{p + 1}$, in $A \oplus C\setminus \{a_1 + c_1, \ldots, a_n + c_p\}$ is equal to $a_1 + c_{p + 1}$. These sets are equal, thus, $b_{p + 1} = c_{p + 1}$; which is a contradiction. 

Similarly, we consider the case when $B \subsetneq C$, i.e. $p = k$ or $p = m$.

Lemma is proved.  
\end{proof}

\begin{theorem} \label{th6}
Suppose $Y$ is an attractor in $\Z$ with the lengths of segments (consecutive integer numbers in $Y$) equal to $s$. Then $Y$ is a sum of several arithmetic progressions. one of which is the progression $\{0, \ldots, s - 1\}$. 
\end{theorem}
\begin{proof}
We prove this theorem by induction on $h$, the minimal number of shifts $Y$ needed to be done to entirely fill the segment of integer numbers. 
\textsl{Base}: if $h = 0$, then we already have the segment that has the length $s$, and it is an arithmetic progression of the required form. 

\textsl{Step}: suppose the statement is proved for all $h_1 < h$, now we show it for $h$. 

If the length of the segments is $s \neq 1$, then we present $Y$ in the form $Y_0 \oplus \{0, \ldots, s - 1\}$, where  the length of all segments in $Y_0$ is equal to one (we can choose as $Y_0$ the beginnings of the segments in $Y$). This presentation as a direct sum is well defined since the distances between the elements of $Y_0$ are at least $s$, and if $y_1, y_2 \in Y_0, y_1 < y_2$, $s_1, s_2 \in \{0, \ldots, s - 1\}$, then $y_1 + s_1 < y_2 + s_2$. $Y_0$ is also an attractor, let $l_0$ be the distance between the first two of its points. It remains to prove that $Y_0$ is a sum of arithmetic progressions. Let us apply Theorem \ref{doubleind} for $Y_0$ with a very large $k$ (for example, larger than the value of a maximal shift). So we get that the set of all shifts $L_0$ is an $l$-set, where $l_0$ is in the role of $l$. From the definition of $l$-set it follows that $L_0 = \{0, 1 \ldots l_0 - 1\} \oplus L_1$, where $L_1$ is the set with the beginnings of segments in $L_0$. Since $L_0$ is the set of shifts, $Y_0 \oplus L_0$ is the segment of integer numbers. We obtain that $Y_0 \oplus \{0, 1 \ldots l_0 - 1\} \oplus L_1$ is the segment of integer numbers. 
 
 Consider $Y_1 = Y_0 \oplus \{0, 1 \ldots l_0 - 1\}$, it is a well-defined set consisting of several segments of integer numbers. It is an attractor, since $Y_1 \oplus L_1$ is a segment of integer numbers, so we can tile the segment of integer numbers by shifting $Y_1$ with the set $L_1$. 
Since all distances between the points in $Y_0$ are multiple of $l_0$, the length of the first segment in $Y_1$ is divisible by $l_0$ (the first segment in $Y_1$ consists of several segments having the length $l_0$), let this length be equal to $p \cdot l_0$.  $L_1$ has $l_0$ times less elements than $L_0$, and $L_0$ has $s$ times elements more than the set of shifts corresponding to $Y$. Since $l_0$ as a distance between first points of $Y_0$ is more than $s$, $L_1$ has less than $h$ elements, therefore we can apply the induction hypothesis for the attractor $Y_1$. Then $Y_1$ is a sum of arithmetic progressions with one of them equal to $\{0, 1, \ldots, p \cdot l_0 - 1\}$. 

Notice that 
$$\{0, 1, \ldots, p \cdot l_0 - 1\} = \{0, 1, \ldots, l_0 - 1\} \oplus \{0, l_0, 2 \cdot l_0, \ldots, (p - 1) \cdot l_0\}.$$
We obtain that $Y_1$ is a sum of arithmetic progressions with one of them equal to $\{0, 1, \ldots, l_0 - 1\}$. 
By construction, $Y_1 = Y_0 \oplus \{0, 1 \ldots l_0 - 1\}$, hence, the lemma on direct sums of set implies that $Y_0$ is a sum of arithmetic progressions, this completes the proof.

Theorem \ref{th6} is proved. 
\end{proof}

Now we specify which arithmetic progressions can be summed up to get attractors. For these purposes we reformulate the problem in the terms of polynomials. Instead of the set of integer numbers $A = \{a_1, \ldots, a_n\}$ we consider the polynomial $(z^{a_1} + \ldots + z^{a_n})$. Then the direct sum of sets corresponds to the product of these polynomials, arithmetic progressions correspond to geometric progressions. 

If $Y$ is an attractor, and $L$ is a set of shifts, then $Y \oplus L$ is a segment of integer numbers. We can also consider $L$ as an attractor, and $Y$ as a set of shifts. Thus, $L$ is also a sum of arithmetic progressions. The total sum of arithmetic progressions is a segment of integer numbers. Our problem is equivalent to the question for which $a_1, d_1 \ldots, a_n, d_n$ there is $N$ such that  

$$1 + z + \cdots + z^{N} = (1 + z^{a_1} + \cdots + z^{a_1 \cdot (d_1 - 1)}) \cdots (1 + z^{a_n} + \cdots + z^{a_n \cdot (d_n - 1)}).$$
After that we can consider the sum of an arbitrary subset of these arithmetic progressions as an attractor, and the sum of the remaining progressions as the set of shifts.  
 
Denote $P_d(z) = 1 + z + \cdots + z^{d - 1}$. 

\begin{lemma} \label{lem6}
In the notation above it holds that  
$$P_{d_1 \cdots d_n}(z) = P_{d_1}(z) \cdot P_{d_2}(z^{d_1}) \cdot P_{d_3}(z^{d_1 \cdot d_2}) \cdots  P_{d_n}(z^{d_1\cdots d_{n - 1}}).$$
\end{lemma}
\begin{proof}
We prove by induction on $k$ that 
$$P_{d_1 \cdots d_k}(z) = P_{d_1}(z) \cdot P_{d_2}(z^{d_1}) \cdots P_{d_k}(z^{d_1\cdots d_{k - 1}}).$$
\textsl{Base}: in case $k = 1$ there is nothing to prove; $k = 2$ follows from the simple check  
$$(1 + z + \ldots + z^{d_1 \cdot d_2 - 1}) =  (1 + z + \ldots + z^{d_1 - 1}) \cdot (1 + z^{d_1} + \ldots + z^{d_1 \cdot (d_2 - 1)}).$$
\textsl{Step}: suppose the statement is established for $k > 1$, let us show it for $k + 1$: 
\begin{gather*}
P_{d_1 \cdots d_{k + 1}}(z) = P_{d_1}(z) \cdot P_{d_2 \cdots d_{k + 1}}(z^{d_1}) = P_{d_1}(z) \cdot P_{d_2}({z^{d_1}}^{1}) \cdots P_{d_{k + 1}}({z^{d_1}}^{d_2 \cdots d_{k}})\\
= P_{d_1}(z) \cdot P_{d_2}(z^{d_1}) \cdots P_{d_{k + 1}}(z^{d_1\cdots d_{k}}).
\end{gather*}
$$P_{d_1 \cdots d_{k + 1}}(z) = P_{d_1}(z) \cdot P_{d_2 \cdots d_{k + 1}}(z^{d_1}) 
= P_{d_1}(z) \cdot P_{d_2}(z^{d_1}) \cdots P_{d_{k + 1}}(z^{d_1\cdots d_{k}}).$$

Lemma is proved. 
\end{proof}

\begin{theorem} Suppose $a_1 \le \cdots \le a_n$. Then 
$$1 + z + \cdots + z^{N} = (1 + z^{a_1} + \cdots + z^{a_1 \cdot (d_1 - 1)}) \cdots (1 + z^{a_n} + \cdots + z^{a_n \cdot (d_n - 1)})$$ 
if and only if $a_1 = 1$, $a_k = d_1 \cdots d_{k - 1}$ for every $k$, $2 \le k \le n$, $N = d_1 \cdots d_n$. 
\end{theorem}
\begin{proof}
The sufficiency follows from Lemma \ref{lem6}. 

We prove the necessity by induction. The equality $a_1 = 1$ follows from the fact that since there is the term $z$ in the left hand side, it is also on the right, and that is impossible if all incoming degrees of $z$ are more than one. Suppose that we already have shown that $a_1 = 1$, $a_i = d_1 \cdots d_{i - 1}$ for every $k$, $\forall 2 \le i \le k$; let us prove it for $k + 1$. By the induction hypothesis 

\begin{gather*}
1 + z + \cdots + z^N = P_{d_1}(z) \cdots P_{d_k}(z^{d_1 \cdots d_{k - 1}}) \\
\times  (1 + z^{a_{k + 1}} + \cdots + z^{a_{k + 1} \cdot (d_{k + 1} - 1)}) \cdots (1 + z^{a_n} + \cdots + z^{a_n \cdot (d_n - 1)}) \\
= P_{d_1 \cdots d_k}(z) \cdot (1 + z^{a_{k + 1}} + \cdots + z^{a_{k + 1} \cdot (d_{k + 1} - 1)}) \cdots (1 + z^{a_n} + \cdots + z^{a_n \cdot (d_n - 1)}).
\end{gather*}
 
If $k \ne n$, i.e. the decomposition is not completed, then in the left hand side there is a term $z^{d_1 \cdots d_k}$. It cannot appear on the right, if from all multipliers except the first one choose the terms equal to one. Thus, the final degree of $z$ will be at least $a_{k + 1}$, then $a_{k + 1} \le d_1 \cdots d_k$. Along with it, in the sum in the right hand side all terms of the form $z^s$, where $s < d_1 \cdots d_k$, are obtained as a multiplication of the first bracket and of several ones. Therefore, they should not be obtained in the right hand side in a different way, so $a_{k + 1} \ge d_1 \cdots d_k$. 

We obtain that $a_{k + 1} = d_1 \cdots d_k$, as required. The formula for $N$ follows then from Lemma \ref{lem6}, thus,  

$$P_N(z) = 1 + z + \cdots + z^{N - 1} = P_{d_1}(z) \cdots P_{d_n}(z^{d_1 \cdots d_{n - 1}}) = P_{d_1 \cdots d_n}(z).$$
Theorem is proved.  
\end{proof}

Thus, we proved the fundamental theorem on the classification of attractors in $\Z$. 

\begin{theorem} \label{mainoned}
An arbitrary one-dimensional attractor consisting of nonnegative integer numbers with the beginning in zero is the sum of a subset of arithmetic progressions from the set 
$$\{0, a_1, \ldots, a_1 \cdot (d_1 - 1)\}, \ldots, \{0, a_n, \ldots, a_n \cdot (d_n - 1)\},$$ 
where $a_1 = 1, \ldots$, $a_k = d_1 \cdots d_{k - 1}$ for every $k$, $2 \le k \le n$, $d_i$ are arbitrary integer numbers $\ge 2$. The lengths of the segments of attractor are one if and only if we do not choose the arithmetic progression $\{0, a_1, \ldots, a_1 \cdot (d_1 - 1)\}$. 
\end{theorem}

\section{Examples of two-dimensional attractors}
\label{examptwo}
In Section \ref{lineres}
we found out the structure of one-dimensional attractors consisting of finitely many segments. Using them it is easy to construct two-dimensional attractors as tensor products of one-dimensional attractors. 

\begin{example}\label{twodimill}
The example of two-dimensional disconnected attractor which was obtained as tensor product of disconnected one-dimensional attractors is given in Fig. \ref{twodim2}. Fig. \ref{twodim} shows the tiling of the plane with this attractor. Each shift is shown with its own colour. 

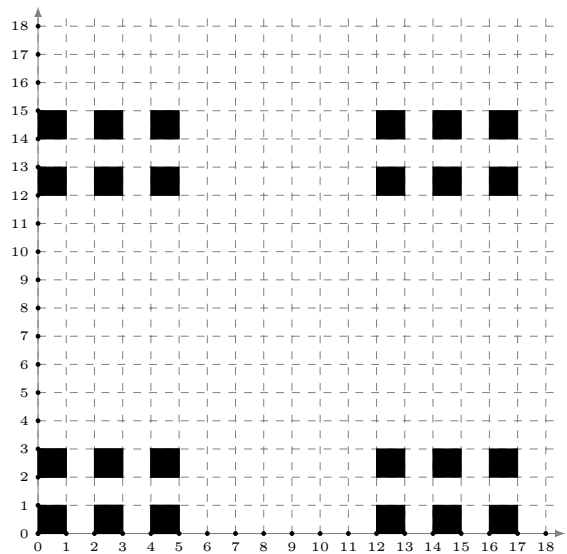
\begin{figure}[ht]
  \centering
  \begin{tikzpicture} [scale = 0.75]
    
    \pgftransformcm{0.5}{0}{0}{0.5}{\pgfpoint{0cm}{0cm}}
    \coordinate (XAxisMin) at (-11, 0);
    \coordinate (XAxisMax) at (7.7, 0);

    \coordinate (YAxisMin) at (-11, 0);
    \coordinate (YAxisMax) at (-11, 18.7);
    \draw [thin, gray,-latex] (XAxisMin) -- (XAxisMax);
	\draw [thin, gray,-latex] (YAxisMin) -- (YAxisMax);

    \draw[style=help lines,dashed] (-11, 0) grid[step=1cm] (7.5, 18.5);
\foreach \x in {-11,-10,...,7}{

\node[draw,circle,inner sep=0.5pt,fill] at (\x, 0) {};
\node[draw,circle,inner sep=0.5pt,fill] at (-11, \x + 11) {};
}
\foreach \x in {0,1,...,18}{
	\node at (\x - 11, 0) [below] {\tiny{\x}};

	\node at (-11, \x) [left] {\tiny{\x}};
    }
\colorlet{c00}{black};

\foreach \i in {0}{
\foreach \j in {0}{
\draw[c\i\j, fill] (-11 + \i, \j) -- (-11 + \i, 1 + \j) --(-10 + \i, 1 + \j) --(-10 + \i, 0 + \j) -- cycle;
\draw[c\i\j, fill] (-11 + \i, 2+\j) -- (-11 + \i, 3 + \j) --(-10 + \i, 3 + \j) --(-10 + \i, 2 + \j) -- cycle;
\draw[c\i\j, fill] (-9 + \i, \j) -- (-9 + \i, 1 + \j) --(-8 + \i, 1 + \j) --(-8 + \i, 0 + \j) -- cycle;
\draw[c\i\j, fill] (-9 + \i, 2+\j) -- (-9 + \i, 3 + \j) --(-8 + \i, 3 + \j) --(-8 + \i, 2 + \j) -- cycle;
\draw[c\i\j, fill] (-7 + \i, \j) -- (-7 + \i, 1 + \j) --(-6 + \i, 1 + \j) --(-6 + \i, 0 + \j) -- cycle;
\draw[c\i\j, fill] (-7 + \i, 2+\j) -- (-7+ \i, 3 + \j) --(-6 + \i, 3 + \j) --(-6 + \i, 2 + \j) -- cycle;

\draw[c\i\j, fill] (1 + \i, \j) -- (1 + \i, 1 + \j) --(2 + \i, 1 + \j) --(2 + \i, 0 + \j) -- cycle;
\draw[c\i\j, fill] (1 + \i, 2+\j) -- (1 + \i, 3 + \j) --(2 + \i, 3 + \j) --(2 + \i, 2 + \j) -- cycle;
\draw[c\i\j, fill] (3 + \i, \j) -- (3 + \i, 1 + \j) --(4 + \i, 1 + \j) --(4 + \i, 0 + \j) -- cycle;
\draw[c\i\j, fill] (3 + \i, 2+\j) -- (3 + \i, 3 + \j) --(4 + \i, 3 + \j) --(4 + \i, 2 + \j) -- cycle;
\draw[c\i\j, fill] (5 + \i, \j) -- (5 + \i, 1 + \j) --(6 + \i, 1 + \j) --(6 + \i, 0 + \j) -- cycle;
\draw[c\i\j, fill] (5 + \i, 2+\j) -- (5+ \i, 3 + \j) --(6 + \i, 3 + \j) --(6 + \i, 2 + \j) -- cycle;

\foreach \k in {-11, -9, -7, 1, 3, 5}{
\draw[c\i\j, fill] (\k + \i, 12 + \j) -- (\k + \i, 13 + \j) --(\k + 1 + \i, 13 + \j) --(\k + 1 + \i, 12 + \j) -- cycle;
\draw[c\i\j, fill] (\k + \i, 14+\j) -- (\k + \i, 15 + \j) --(\k + 1 + \i, 15 + \j) --(\k + 1 + \i, 14 + \j) -- cycle;
}
}}
\draw[black, fill] (-11, 2) -- (-11, 3) --(-10, 3) --(-10, 2) -- cycle;
\draw[black, fill] (-9, 0) -- (-9, 1) --(-8, 1) --(-8, 0) -- cycle;
\draw[black, fill] (-9, 2) -- (-9, 3) --(-8, 3) --(-8, 2) -- cycle;
\draw[black, fill] (-7, 0) -- (-7, 1) --(-6, 1) --(-6, 0) -- cycle;
\draw[black, fill] (-7, 2) -- (-7, 3) --(-6, 3) --(-6, 2) -- cycle;
    \end{tikzpicture}
  \caption{Two-dimensional disconnected attractor.}
\label{twodim2}
\end{figure}

\begin{figure}[ht]
  \centering
  \begin{tikzpicture} [scale = 0.75]
    
    \pgftransformcm{0.5}{0}{0}{0.5}{\pgfpoint{0cm}{0cm}}
    \coordinate (XAxisMin) at (-11, 0);
    \coordinate (XAxisMax) at (13.7, 0);

    \coordinate (YAxisMin) at (-11, 0);
    \coordinate (YAxisMax) at (-11, 24.7);
    \draw [thin, gray,-latex] (XAxisMin) -- (XAxisMax);
	\draw [thin, gray,-latex] (YAxisMin) -- (YAxisMax);

    \draw[style=help lines,dashed] (-11, 0) grid[step=1cm] (13.5, 24.5);
\foreach \x in {-11,-10,...,13}{

\node[draw,circle,inner sep=0.5pt,fill] at (\x, 0) {};
\node[draw,circle,inner sep=0.5pt,fill] at (-11, \x + 11) {};
}
\foreach \x in {0,1,...,24}{
	\node at (\x - 11, 0) [below] {\tiny{\x}};

	\node at (-11, \x) [left] {\tiny{\x}};
    }
\colorlet{c00}{red!80!white};%black
\colorlet{c01}{yellow!30!pink};%red!70!white
\colorlet{c10}{yellow!60!white};%yellow!50!orange%green
\colorlet{c11}{orange!90!black};%yellow!40!orange%green!70!red%blue!70!black
\colorlet{c60}{violet!70!blue};%orange
\colorlet{c61}{cyan!80!white};%cyan
\colorlet{c70}{blue!40!white};%magenta
\colorlet{c71}{cyan!40!blue};%yellow

\colorlet{c04}{black!70!green};%pink
\colorlet{c14}{green};%yellow!40!blue
\colorlet{c05}{olive!80!white};%cyan!50!blue
\colorlet{c15}{green!50!black};%yellow!50!orange

\colorlet{c64}{magenta!50!violet};%magenta!50!violet
\colorlet{c74}{orange!30!pink};%orange!50!green
\colorlet{c65}{pink}%olive
\colorlet{c75}{pink!60!blue};%blue!50!white

\colorlet{c08}{blue!50!black};%black!30
\colorlet{c18}{cyan!45!green}%green!30
\colorlet{c09}{cyan!30};%red!30
\colorlet{c19}{black!70!white}%blue!30

\colorlet{c68}{red!50!black}%orange!30;
\colorlet{c78}{yellow!90!black}%magenta!30;
\colorlet{c69}{yellow!30};
\colorlet{c79}{brown!85!black};

\foreach \i in {0, 1, 6, 7}{
\foreach \j in {0, 1, 4, 5, 8, 9}{
\draw[c\i\j, fill] (-11 + \i, \j) -- (-11 + \i, 1 + \j) --(-10 + \i, 1 + \j) --(-10 + \i, 0 + \j) -- cycle;
\draw[c\i\j, fill] (-11 + \i, 2+\j) -- (-11 + \i, 3 + \j) --(-10 + \i, 3 + \j) --(-10 + \i, 2 + \j) -- cycle;
\draw[c\i\j, fill] (-9 + \i, \j) -- (-9 + \i, 1 + \j) --(-8 + \i, 1 + \j) --(-8 + \i, 0 + \j) -- cycle;
\draw[c\i\j, fill] (-9 + \i, 2+\j) -- (-9 + \i, 3 + \j) --(-8 + \i, 3 + \j) --(-8 + \i, 2 + \j) -- cycle;
\draw[c\i\j, fill] (-7 + \i, \j) -- (-7 + \i, 1 + \j) --(-6 + \i, 1 + \j) --(-6 + \i, 0 + \j) -- cycle;
\draw[c\i\j, fill] (-7 + \i, 2+\j) -- (-7+ \i, 3 + \j) --(-6 + \i, 3 + \j) --(-6 + \i, 2 + \j) -- cycle;

\draw[c\i\j, fill] (1 + \i, \j) -- (1 + \i, 1 + \j) --(2 + \i, 1 + \j) --(2 + \i, 0 + \j) -- cycle;
\draw[c\i\j, fill] (1 + \i, 2+\j) -- (1 + \i, 3 + \j) --(2 + \i, 3 + \j) --(2 + \i, 2 + \j) -- cycle;
\draw[c\i\j, fill] (3 + \i, \j) -- (3 + \i, 1 + \j) --(4 + \i, 1 + \j) --(4 + \i, 0 + \j) -- cycle;
\draw[c\i\j, fill] (3 + \i, 2+\j) -- (3 + \i, 3 + \j) --(4 + \i, 3 + \j) --(4 + \i, 2 + \j) -- cycle;
\draw[c\i\j, fill] (5 + \i, \j) -- (5 + \i, 1 + \j) --(6 + \i, 1 + \j) --(6 + \i, 0 + \j) -- cycle;
\draw[c\i\j, fill] (5 + \i, 2+\j) -- (5+ \i, 3 + \j) --(6 + \i, 3 + \j) --(6 + \i, 2 + \j) -- cycle;

\foreach \k in {-11, -9, -7, 1, 3, 5}{
\draw[c\i\j, fill] (\k + \i, 12 + \j) -- (\k + \i, 13 + \j) --(\k + 1 + \i, 13 + \j) --(\k + 1 + \i, 12 + \j) -- cycle;
\draw[c\i\j, fill] (\k + \i, 14+\j) -- (\k + \i, 15 + \j) --(\k + 1 + \i, 15 + \j) --(\k + 1 + \i, 14 + \j) -- cycle;
}
}}
    \end{tikzpicture}
  \caption{The tiling of the plane with a two-dimensional disconnected attractor.}
\label{twodim}
\end{figure}
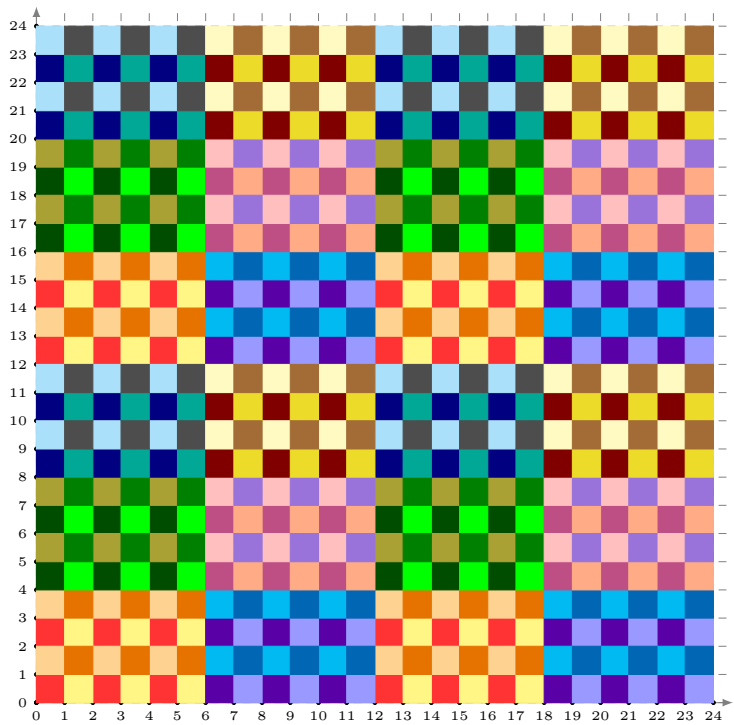

\end{example}

\begin{example}
Let us notice that on the plane there exist disconnected tiles that are not polygonal sets. Moreover, such tiles could have fractal properties: irregular boundaries, etc. One of possible examples is given in Fig. \ref{frac}. It is obtained using the software \cite{KM}. 

\begin{figure}[ht]
\centering
\includegraphics[width = 0.5\textwidth]{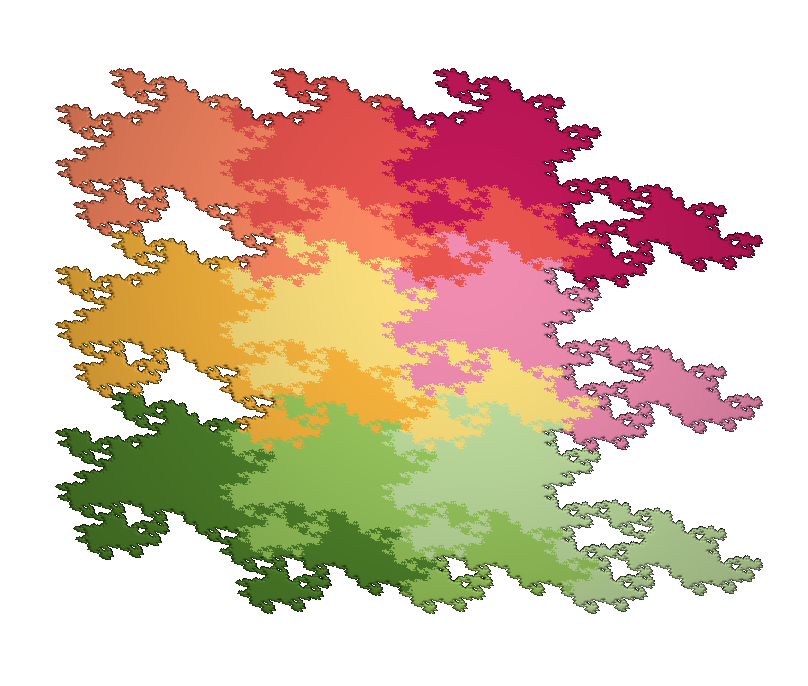}
\caption{The example of disconnected fractal attractor on the plane.}
\label{frac}
\end{figure} 

\end{example}

\section{Conjectures and possible generalizations}
\label{hyp}
Let us summarize the results of this work. We have obtained the full classification of box-attractors in all dimensions, that is, we have classified all integer dilation matrices $M$ and corresponding sets of shifts which generate attractors-parallelepipeds. (Theorems \ref{mainbox} and \ref{tbox2}).  Then we have shown that every convex attractor is a parallelepiped (Theorem \ref{mainconvex}). If we look for attractors among not necessarily convex polygons, we obtain due to Theorem \ref{mainnonconvex} that in the plane they could be only parallelograms.  For higher dimensions we have not obtained a generalization and we formulate it as a conjecture: 

\begin{conjecture} \label{hyppp}
Every, not necessary convex,  attractor in $\R^d$ which is a polygon is a parallelepiped. 
\end{conjecture}

The proof of Theorem \ref{mainnonconvex} cannot be straightforwardly generalized for higher dimensions since in arbitrary dimension the intersections of dilated copies of the set with the initial set can be much more complicated. They even do not need to be connected. For example, in Fig. \ref{multidim} we can see nonconvex polyhedron $G$ resting on the plane by two faces. According to the terminology from the proof of Theorem \ref{mainnonconvex}, each of these two faces is extreme for the polygon $G$. The plane which contains them can be considered as a face of the contracted polygon $G$. In Fig. \ref{complexity1}, \ref{complexity2} it is shown that the polyhedron $G$ can be shifted parallel to this plane so that the shifted set will not intersect the set $G$ and they will adjoin to each other.

\begin{figure}[ht]
\centering
\includegraphics[width = 0.5\textwidth]{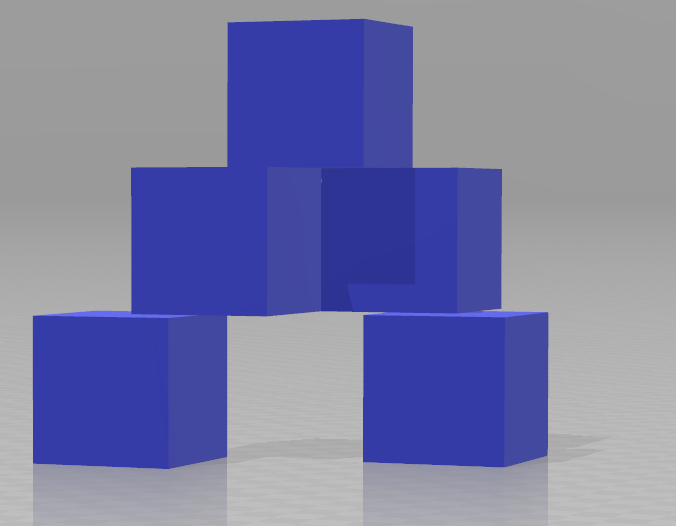}
\caption{Polyhedron resting on the plane by two faces.}
\label{multidim}
\end{figure}

\begin{figure}[ht]
\begin{minipage}[h]{0.45\linewidth}
\centering
\includegraphics[width = 1\textwidth]{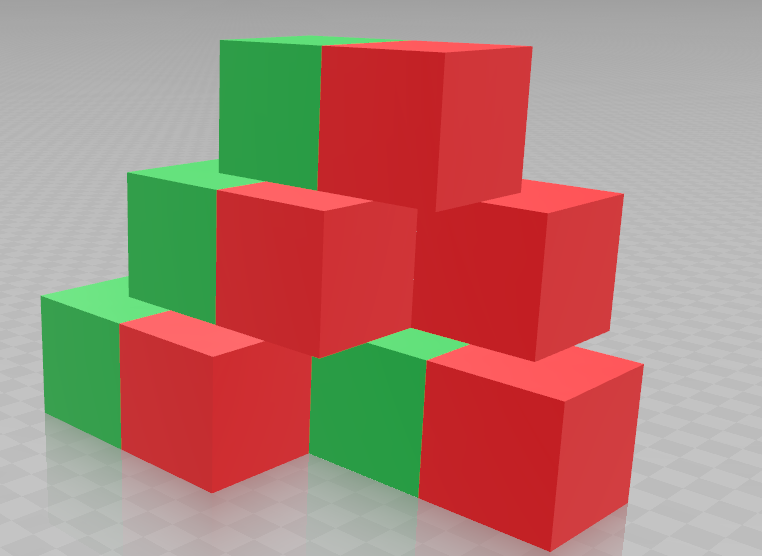}

\caption{Polyhedron resting on the plane by two faces and its shift.}
\label{complexity1}
\end{minipage}
\quad
\begin{minipage}[h]{0.45\linewidth}
\centering
\includegraphics[width = 1\textwidth]{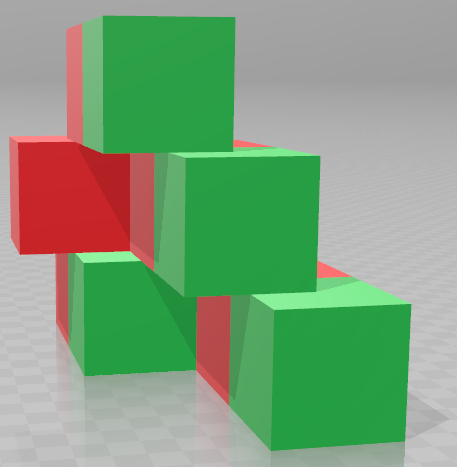}

\caption{Polyhedron resting on the plane by two faces and its shift (another view).}
\label{complexity2}
\end{minipage}
\end{figure}

Let us also remark that for Conjecture  \ref{hyppp} and Theorems \ref{mainbox}--\ref{mainoned} it is important that we deal with self-similar sets (attractors, tiles). Polyhedra whose shifts can tile the space can be nontrivial. For example, in Fig. \ref{nontrivial} it is shown a polyhedron whose shifts can tile the layer in $\R^3$ and so the whole $\R^3$ (the tiling is shown in Fig.\ref{interestingtiling}).

\begin{figure}[h!]
\begin{minipage}[h!]{0.45\linewidth}
\centering
\includegraphics[width = 1\textwidth]{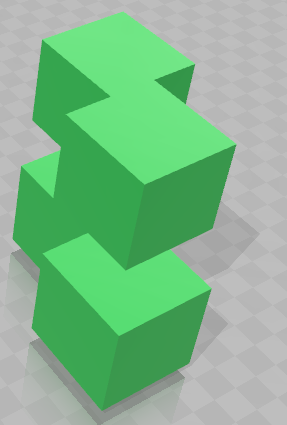}

\caption{Not self-similar polyhedron whose shifts tile the space.}
\label{nontrivial}
\end{minipage}
\quad
\begin{minipage}[h!]{0.45\linewidth}
\centering
\includegraphics[width = 1\textwidth]{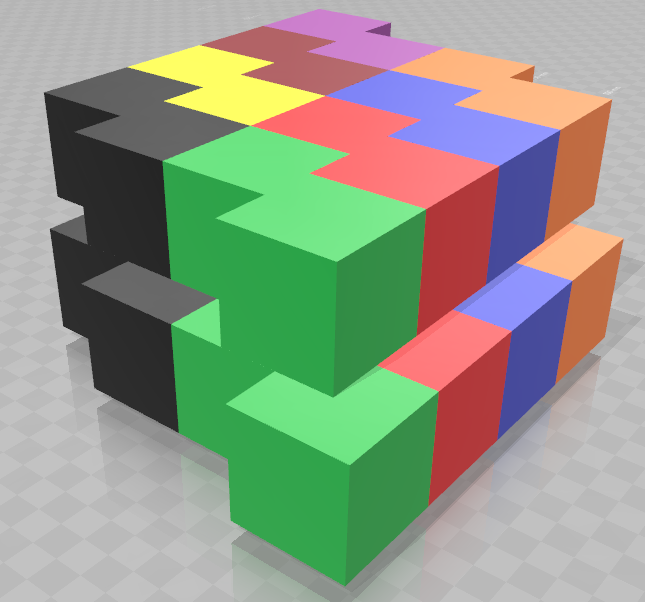}

\caption{Layer tiling with a nonconvex polyhedron.}
\label{interestingtiling}
\end{minipage}
\end{figure}
 
On the other hand, simple attractors different from parallelepipeds do exist. Examples \ref{onedimexam}, \ref{twodimill} are constructed as the special sets of parallelepipeds. 
We have proved the comleteness of this characterization only on the line, for the systems of segments (Theorem \ref{mainoned}). For two-dimensional case, let us formulate the following conjecture: 

\begin{conjecture} \label{hyp2dim} Every two-dimensional attractor consisting of finitely many polygons is affinely similar to a tensor product of one-dimensional attractors.
\end{conjecture}
Similarly, we can formulate Conjectrure \ref{hyp2dim} in an arbitrary dimension.

Let us notice that we investigated not all sets in $\Z$ whose integer shifts tile the line, but only attractors. Because of their self-similarity their shifts tile some segment of the line (Lemma \ref{to_tiling} and Remark \ref{add_remark}). They are automatically periodic since further we periodically shift filled segment along the whole line. In the general case, if considering all sets whose integer shifts cover the line, then they also are periodic, it is shown, for example, in \cite{Kol}. Ibidem, in \cite[p. $3.3$]{Kol}, the examples of non-periodic tilings in higher dimensions are given. 
However, the structure of an arbitrary set which can cover the line with integer shifts is unknown (see \cite[p. $1$]{Kol}). 

\section{Acknowledgements}
The author is grateful to her advisor V.Yu.~Protasov for permanent support and his help in work, and also to anonymous reviewers for their valuable comments which improved this text.

\end{document}